\documentclass[a4paper]{amsart}
\usepackage{epsfig}
\usepackage{amssymb}
\usepackage{amsmath}
\usepackage{latexsym}

\usepackage{color}
\usepackage{graphicx}
\usepackage{amsthm}
\usepackage[all]{xy}

\usepackage{color}
\newtheorem{thm}{Theorem}[section] \newtheorem{lem}[thm]{Lemma}
 
\newtheorem{prop}[thm]{Proposition} \theoremstyle{definition}
\newtheorem{defn}[thm]{Definition}

\newtheorem{conv}[thm]{Convention} \newtheorem{rem}[thm]{Remark}

\newtheorem{theor}{Theorem}

\definecolor{darkred}{rgb}{1,0,0} 
\definecolor{darkgreen}{rgb}{0,0.8,0}
\definecolor{darkblue}{rgb}{0,0,1}

\begin{document}
\title{Nielsen equivalence in a class of random groups}

\author{Ilya Kapovich and Richard Weidmann}

\address{\tt Department of Mathematics, University of Illinois at
 Urbana-Champaign, 1409 West Green Street, Urbana, IL 61801, USA
 \newline http://www.math.uiuc.edu/\~{}kapovich/} \email{\tt
  kapovich@math.uiuc.edu}

\address{\tt Mathematisches Seminar, Christian-Albrechts-Universit\"at zu Kiel, Ludewig-Meyn Str.~4, 24098 Kiel, Germany}
\email{\tt weidmann@math.uni-kiel.de}


\thanks{The first author was supported by he Collaboration Grant no. 279836 from the Simons Foundation, and  by the National Science Foundation grant DMS-1405146}


\date{}

\maketitle

\begin{abstract} 
We show that for every $n\ge 2$ there exists a torsion-free one-ended word-hyperbolic group $G$ of rank  $n$ admitting generating $n$-tuples $(a_1,\ldots ,a_n)$ and $(b_1,\ldots ,b_n)$ such that the $(2n-1)$-tuples $$(a_1,\ldots ,a_n, \underbrace{1,\ldots ,1}_{n-1 \text{ times}})\hbox{ and }(b_1,\ldots, b_n, \underbrace{1,\ldots ,1}_{n-1 \text{ times}} )$$ are not Nielsen equivalent in $G$. The group $G$ is produced via a probabilistic construction.
\end{abstract}

\section*{Introduction}

The notion of Nielsen equivalence goes back to the origins of geometric group theory and the work of Jakob Nielsen in nineteen twenties \cite{N1,N2}.
If $G$ is a group,  $n\ge 1$ and $\tau=(g_1,\dots, g_n)$ is an  $n$-tuple of elements of $G$,  an \emph{elementary Nielsen transformation}  on $\tau$ is one of the following three types of moves:
\begin{enumerate}
\item For some $i\in \{1,\dots, n\}$ replace $g_i$ in $\tau$ by $g_i^{-1}$.
\item For some $i\ne j$, $i,j\in \{1,\dots, n\}$ interchange $g_i$ and $g_j$ in $\tau$. 
\item For some $i\ne j$, $i,j\in \{1,\dots, n\}$ 
replace $g_i$ by $g_ig_j^{\pm 1}$ in~$\tau$.
\end{enumerate}
Two $n$-tuples $\tau=(g_1,\dots, g_n)$ and $\tau=(g_1',\dots, g_n')$ of elements of $G$ are called \emph{Nielsen equivalent}, denoted $\tau\sim_{NE} \tau'$, if there exists a finite chain of elementary Nielsen transformations taking $\tau$ to $\tau'$.   Since elementary Nielsen transformations are invertible it follows that Nielsen equivalence is an equivalence relation on the set $G^n$ of  $n$-tuples of elements of $G$ for every $n\ge 1$. 

\smallskip Nielsen showed~\cite{N1,N2}  that for the free group $F_n:=F(x_1,\ldots ,x_n)$ any generating $n$-tuple of $F_n$ is Nielsen equivalent to the standard basis $(x_1,\ldots ,x_n)$. This fact implies, in particular, that the automorphism group of a finitely generated free group is generated by so called ``Nielsen automorphisms" that correspond to the elementary equivalences discussed above. It follows that for any group $G$ two $n$-tuples $\tau=(g_1,\dots, g_n), \tau'=(g_1',\dots, g_n')\in G^n$ are Nielsen equivalent  in $G$ if and only if there exist $\varphi\in Hom(F_n,G)$ and $\alpha\in Aut(F_n)$ such that the following hold:
\begin{enumerate}
\item $g_i=\varphi(x_i)$ for $1\le i\le n$.
\item $g_i'=\varphi\circ\alpha(x_i)$ for $1\le i\le n$.
\end{enumerate}
Thus Nielsen equivalence classes of $n$-tuples in $G$ correspond to orbits in $Hom(F_n,G)$ under the natural right action of $Aut(F_n)$. In the case of generating tuples these are orbits in $Epi(F_n,G)$ under the action of $Aut(F_n)$.

\smallskip If $\langle g_1,\dots, g_n\rangle=G$ then an arbitrary $h\in G$ can be expressed as a product of elements of $\{g_1,\dots, g_n\}^{\pm 1}$.  Therefore the $(n+1)$-tuples $(g_1,\dots,g_n,1)$ and $(g_1,\dots,g_n,h)$ are Nielsen-equivalent. For similar reasons, if $k\ge 1$ and $h_1,\dots, h_k\in G$ are arbitrary, then $(g_1,\dots,g_n,\underbrace{1,\dots, 1}_{k \text{ times}})\sim_{NE} (g_1,\dots, g_n, h_1,\dots, h_k)$.  It follows that if $(h_1,\dots, h_n)\in G$ is another generating tuple for $G$ then
\[
(g_1,\dots,g_n,\underbrace{1,\dots, 1}_{n \text{ times}})\sim_{NE} (g_1,\ldots ,g_n,h_1,\ldots,h_n)\sim_{NE} (h_1,\dots,h_n,\underbrace{1,\dots, 1}_{n \text{ times}}) \quad \text{ in } G.
\]

Recall that for a finitely generated group $G$ the \emph{rank} of $G$, denoted $rank(G)$, is the smallest number of elements in a generating set of $G$. A generating $n$-tuple of $G$ is called \emph{minimal} if $n=rank(G)$. 

The operation of transforming an $n$-tuple of elements of $G$ into an $(n+1)$-tuple by appending the identity element as the $(n+1)$-st entry  is sometimes called a \emph{stabilization move}.  As we observed above,  any two generating $n$-tuples of a group $G$ become Nielsen equivalent after $n$ stabilization moves.  The main goal of this paper is to show that there exist groups that have two generating $n$-tuples that do not become Nielsen equivalent after $n-1$ stabilizations. Evans previously showed~\cite{E2,E3} that for any $k$ there exists a metabelian group $G$ and generating $n$-tuples that do not become Nielsen equivalent after $k$ stabilizations. However, in these example $n$ is much larger than $k$ and the generating tuples  are not minimal.

\smallskip The main result of this paper is:

\begin{theor}\label{thm:A} Let  $n\ge 2$ be an arbitrary integer. Then there exists a generic set $\mathcal R$ of $2n$-tuples $$\tau=(v_1,\dots,v_n,u_1,\dots,u_n)$$ where for each $\tau\in \mathcal R$ every $v_i$ is a freely reduced  word in $F(a_1,\dots, a_n)$, every $u_i$ is a freely reduced word in $F(b_1,\dots,b_n)$,  where $|v_1|=\dots=|v_n|=|u_1|=\dots=|u_n|$ and such that for each $\tau=(v_1,\dots,v_n,u_1,\dots,u_n)\in \mathcal R$ the following holds:

Let $G$ be a group given by the presentation
\[
G=\langle a_1,\ldots ,a_n,b_1,\ldots ,b_n|a_i=u_i(\underline{b}),b_i=v_i(\underline{a}), \text { for } i=1,\dots, n\rangle. \tag{*}
\]
Then $G$ is a torsion-free one-ended word-hyperbolic group with $rank(G)=n$, and  the $(2n-1)$-tuples $(a_1,\ldots ,a_n,\underbrace{1,\ldots ,1}_{n-1 \text{ times}})$ and $(b_1,\ldots ,b_n,\underbrace{1,\ldots ,1}_{n-1 \text{ times}})$ are not Nielsen-equivalent in $G$.
\end{theor}

The precise meaning of $\mathcal R$ being generic is explained in Section~\ref{S:generic}.  However, the main point of Theorem~\ref{thm:A} is to prove the existence of groups $G$ with the above properties. Genericity considerations are used as a tool for justifying existence of tuples $\tau$ such that $G$ given by presentation (*) satisfies the conclusion of Theorem~\ref{thm:A}.
In an earlier paper~\cite{KW} we proved, using rather different arguments from those deployed in the present article, a weaker related statement. Namely, we proved that if $G$ is given by  presentation (*) corresponding to a generic tuple $\tau=(v_1,\dots, v_n, u_1,\dots, u_n)$, then the $(n+1)$-tuples $(a_1,\dots,a_n,1)$ and  $(b_1,\dots,b_n,1)$ are not Nielsen equivalent in $G$.
In \cite{KW} we also conjectured that  Theorem~\ref{thm:A} should hold.

\smallskip It is in general quite difficult to distinguish Nielsen equivalence classes of tuples of elements. Even in the algorithmically nice setting of torsion-free
word-hyperbolic groups the problem of deciding if two tuples are
Nielsen equivalent is algorithmically undecidable.  Indeed, the subgroup membership problem is a special case of this problem, as
two tuples $(g_1,\ldots,g_n,h)$ and $(g_1,\ldots,g_n,1)$ are Nielsen
equivalent if and only if $h\in\langle g_1,\ldots, g_n\rangle$. Therefore  Nielsen equivalence  is not decidable for finitely
presented torsion-free small cancellation groups, since in general, by a result of Rips~\cite{R},  such groups have
undecidable subgroup membership problem. Only in the case of 2-tuples there is a test that works reasonably well for distinguishing Nielsen equivalence: If two pairs $(g_1,g_2)$ and $(h_1,h_2)$ are Nielsen equivalent in a group $G$, then, by a result of Nielsen \cite{N1}, the commutator $[g_1,g_2]$ is conjuate to $[h_1,h_2]$ or two $[h_1,h_2]^{-1}$. Thus having a solution to the conjugacy problem often gives a way to distinguish Nielsen classes of pairs. This test is particularly important in the theory of finite groups.

\smallskip As noted above, there are no general  methods that allow us to distinguish Nielsen classes. However, Lustig and Moriah \cite{LM1,LM2,LM3} and Evans \cite{E1,E2,E3} produced algebraic methods that work in many special situations. While the methods of Lustig and Moriah only work in the context of minimal generating tuples, the methods of Evans also work for non-minimal and stabilized generating tuples; see the discussion before the statement of Theorem~\ref{thm:A}.  The methods of Moriah and Lustig are $K$-theoretic in nature. The approach of Evans generalizes some linear algebra considerations in the context of solvable groups.  Our approach, sketched in the next section, is essentially geometric. Negatively curved features of groups $G$, given by presentation $(*)$ in Theorem~\ref{thm:A}, play a crucial role in the proof of this theorem.

We are grateful to the two referees of this paper for useful suggestions.

\section{Outline of the proof}

In this section we outline the idea of the proof of Theorem~\ref{thm:A}. First, we pass from presentation (*) of $G$ to the presentation
\[
G=\langle a_1,\dots, a_n| U_1,\dots, U_n\rangle \tag{**}
\]
where $U_i$ is obtained by freely and cyclically reducing the word $u_i(v_1,\dots,v_n)a_i^{-1}$ in $F(A)$, where $A=\{a_1,\dots,a_n\}$.  Genericity of (*) implies that presentation (**) satisfies an arbitrarily strong small cancellation condition $C'(\lambda)$, where $\lambda\in (0,1)$ is an arbitrarily small fixed number. 
Suppose that the $(2n-1)$-tuples $(a_1,\ldots ,a_n,\underbrace{1,\ldots ,1}_{n-1 \text{ times}})$ and $(b_1,\ldots ,b_n,\underbrace{1,\ldots ,1}_{n-1 \text{ times}})$ are Nielsen-equivalent in $G$.
Then in the free group $F(A)$ the $(2n-1)$-tuple $(a_1,\ldots ,a_n,\underbrace{1,\ldots ,1}_{n-1 \text{ times}})$  is Nielsen equivalent to a tuple of freely reduced $F(A)$-words $(x_1,\dots, x_{2n-1})$ such that $x_i=_G b_i$ for $i=1,\dots, n$.  Among all such tuples $(x_1,\dots, x_{2n-1})$ we choose one of minimal complexity, in the appropriate sense. 
We then consider the labelled graph $\Gamma$ given by a wedge of $2n-1$ loops labelled by the words $x_1,\dots, x_{2n-1}$. The fact that $(a_1,\ldots ,a_n,\underbrace{1,\ldots ,1}_{n-1 \text{ times}})\sim_{NE}(x_1,\dots, x_{2n-1})$ in $F(A)$ implies that the graph $\Gamma$ admits a finite sequence of Stallings folds taking it to the graph $R_n$ which is the wedge of $n$ loop-edges labelled $a_1,\dots, a_n$.

\begin{figure}[ht]
	\centering
	\scalebox{.7}{\input{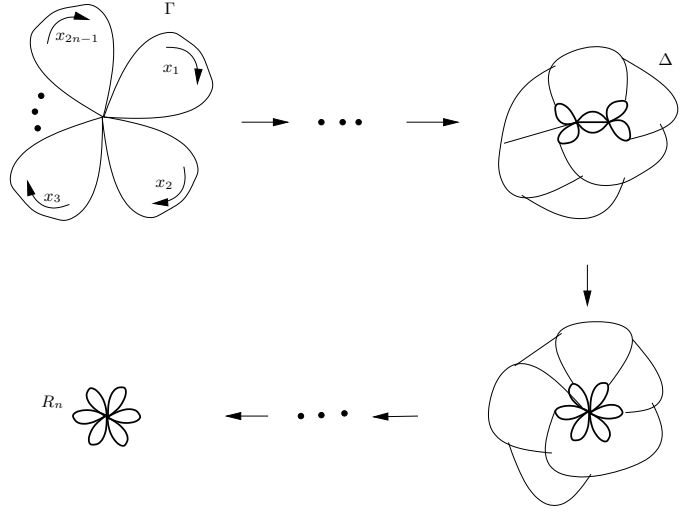}}
	\caption{$\Delta$ is the last graph in the folding sequence that doesn't contain a copy of $R_n$. $\Psi$ is the fat subgraph of $\Delta$.}
\end{figure} 

In this folding sequence we look at the last term $\Delta$ such that $\Delta$ does not contain a copy of $R_n$ as a subgraph. Then $\Gamma$ folds onto $\Delta$ and $\Delta$ contains a subgraph $\Psi$ consisting of $\le n+2$ edges such that $\Psi$ transforms into a graph containing $R_n$ by a single fold.  By construction, the first Betti number of $\Delta$ is $\le 2n-1$, which puts combinatorial constrains on the topology of $\Delta$ which we later exploit. The $F(A)$-words $x_1,\dots x_n$ are readable in $\Gamma$ and therefore they are also readable in~$\Delta$. 

\smallskip In Section~\ref{charrose} and Section~\ref{liftingrandompaths} we study how generic words can be read in labeled graphs of Betti number at most $2n-1$, in particular we show that a path reading such a word only travels few edges more than once unless the graph contains $R_n$ or a 2-sheeted cover of $R_n$, see Theorem~\ref{readingranndomwordsingraphs}. This gives us control over how generic words are read in~$\Delta$. 

\smallskip Using genericity and small cancellation theory considerations and the fact that $x_i=_G b_i$ for $1\le i\le n$, we conclude (see Section~\ref{S:generic}) that  generically one of the following holds: 
\begin{enumerate}
\item $x_i=v_i$ for $1\le i\le n$.
\item For some $i\in\{1,\dots, n\}$ the word $x_i$ contains a subword $Z$ which is almost all of a cyclic permutation $U_*$ of some $U_j^{\pm 1}$.
\end{enumerate}

In Section~\ref{proofmaintheorem} we show that both cases yield a contradiction to either the genericity of the presentation or the minimality assumption mentioned above.

\smallskip In case (1), genericity implies that for each $i$ the path $\gamma_{v_i}$ in $\Delta$, that reads the word $v_i$, runs over a long (subdivided) segment of $\Delta\backslash\Psi$ that is not visited by $\gamma_{v_j}$ for $j\neq i$. Therefore the Betti number of $\Delta/\Psi$ ist at least $n$, which is impossible as $b(\Delta)\le 2n-1$ and $b(\Psi)\ge n$.

\smallskip In case (2), the subword $Z$ in the label of the $i$-th petal of $\Gamma$ projects to a path $s_Z$ with label $Z$ in $\Delta$. Since words readable in $\Psi$ are highly non-generic, we exploit the genericity properties of the presentation to conclude that there is some "long" (comprising a definite portion of the length of $U_j$) arc $Q'$ in $\Delta\setminus \Psi$ with label $W$ such that $s_Z$ runs over $Q'$ exactly once. We then perform a surgery on $\Delta$ by replacing $Q'$ in $\Delta$ by an arc labelled by the word complimentary to $W$ in $U_*$ word $V$ (so that $U_*\equiv WV^{-1}$ in $F(A)$), to obtain a graph $\Delta'$.  Genericity considerations allow us to conclude that  the full pre-image of $Q'$ in $\Gamma$ consists of a disjoint collections of arcs labelled $W$ inside the petals of $\Gamma$. We replace each such arc by an arc labelled $V$  to obtain a graph $\Gamma'$ which  is a wedge of $2n-1$ circles labeled $z_1,\ldots,z_{2n-1}$. Let  $x_i'$ be the word obtained from $z_i$ by free reduction. By construction the graph $\Gamma'$ still folds onto $\Delta'$. Since $\Delta'$ still contains the subgraph $\Psi$, it follows that $\Delta'$ folds onto $R_n$ and hence $\Gamma'$ folds onto $R_n$ as well. Therefore $(x_1',\dots,x_{2n-1}')$  is Nielsen equivalent in $F(A)$ to $(a_1,\ldots ,a_n,\underbrace{1,\ldots ,1}_{n-1 \text{ times}})$.  We then use results established in Section~\ref{reductionmoves} to show that the tuple  $(x_1',\dots,x_{2n-1}')$ has strictly smaller complexity than the tuple $(x_1,\dots, x_{2n-1})$, yielding a contradiction to the minimal choice of $(x_1,\dots, x_{2n-1})$.

\section{Preliminaries}

\subsection{Conventions regarding graphs}

By a \emph{graph} $\Gamma$ we mean a 1-dimensional cell complex,
equipped with the weak topology. We
refer to $0$-cells of $\Gamma$ as \emph{vertices} and to open 1-cells of
$\Gamma$ as \emph{topological edges}. The set of all vertices of
$\Gamma$ is denoted $V\Gamma$ and the set of topological edges of
$\Gamma$ is denoted $E_{top}\Gamma$. For each topological edge
$\mathbf e$ of $\Gamma$ we fix a \emph{characteristic map}
$\tau_\mathbf e: [0,1]\to \Gamma$ which is the attaching map for
$\mathbf e$. Thus $\tau_\mathbf e$ is a continuous map whose
restriction to $(0,1)$ is a homeomorphism between $(0,1)$ and the open
1-cell $\mathbf e$, and the points $\tau_\mathbf e(0), \tau_\mathbf
e(1)$ are (not necessarily distinct) vertices of $\Gamma$.

Each topological edge $\mathbf e$ is homeomorphic to $(0,1)$ and thus,
being a connected 1-manifold, it admits exactly two possible orientations. An \emph{oriented edge} or
just \emph{edge} of $\Gamma$ is a topological edge together with the
choice of an orientation on it. We denote by $E\Gamma$ the set of all
oriented edges of $\Gamma$. If $e\in E\Gamma$ is an oriented edge, we
denote by $e^{-1}\in E\Gamma$ the same topological edge with the
opposite orientation. Thus ${}^{-1}:E\Gamma\to E\Gamma$ is a
fixed-point-free involution on $E\Gamma$. For every oriented  edge $e\in
E\Gamma$ the attaching map for the underlying topological edge
naturally defines the \emph{initial vertex} $\alpha(e)\in V\Gamma$ and
the \emph{terminal vertex} $\omega(e)\in V\Gamma$ of $\Gamma$. Note
that for every $e\in E\Gamma$ we have $\alpha(e)=\omega(e^{-1})$ and
$\omega(e)=\alpha(e^{-1})$.

For a vertex $v\in V\Gamma$ the \emph{degree} $\deg(v)$ of $v$ in
$\Gamma$ is the cardinality of the set $\{e\in E\Gamma| \alpha(e)=v\}$.

Let $\Gamma$ and $\Gamma'$ be graphs. A \emph{morphism} or a
\emph{graph-map} $f:\Gamma\to \Gamma'$ is a continuous map $f$ such
that $f(V\Gamma)\subseteq V\Gamma'$ and such that the restriction of
$f$ to any topological edge (that is, an open 1-cell) $\mathbf e$ of $\Gamma$ is a
homeomorphism between $\mathbf e$ and some topological edge $\mathbf
e'$ of $\Gamma'$. Thus a morphism $f:\Gamma\to \Gamma$ naturally
defines a function (still denoted by $f$) $f:E\Gamma\to
E\Gamma'$. Moreover, for any $e\in E\Gamma$ we have
$f(e^{-1})=(f(e))^{-1}\in E\Gamma'$ and $\alpha(f(e))=f(\alpha(e))$, $\omega(f(e))=f(\omega(e))$.

\subsection{Paths, coverings and path-surjective maps}\label{subsect:paths}

A \emph{combinatorial edge-path} or just an \emph{edge-path} of
\emph{length} $n\ge 1$ in a graph $\Gamma$ is a sequence $\gamma=e_1,\dots,
e_n$ where $e_i\in E\Gamma$ for $i=1,\dots, n$ and such that for all $1\le
i <n$ we have $\omega(e_i)=\alpha(e_{i+1})$.  We denote the length $n$
of $\gamma$ as $n=|\gamma|$ and also denote $\alpha(\gamma):=\alpha(e_1)$ and 
$\omega(\gamma):=\omega(e_n)$.
Also, for any vertex $v\in V\Gamma$ we regard $\gamma=v$ as an edge-path of
length $|\gamma|=0$ with $\alpha(\gamma)=\omega(\gamma)=v$. 
An edge-path $\gamma$ in $\Gamma$ is called \emph{reduced} 
if $\gamma$ does not contain a subpath of the form $e,e^{-1}$ where $e\in
E\Gamma$.

For a finite graph $\Gamma$ (not necessarily connected), we denote by
$b(\Gamma)$ the first betti number of $\Gamma$.
If $\Gamma$ is a graph and $\Gamma'\subseteq \Gamma$ is a subgraph, we
denote by $\Gamma/\Gamma'$ the graph obtained from $\Gamma$ by
collapsing every connected component of $\Gamma'$ to a vertex.

We will need the following elementary but useful lemma whose proof is
left to the reader:

\begin{lem}\label{lem:quotient}
Let $\Gamma$ be a finite graph and let $\Gamma'\subseteq \Gamma$ be a
subgraph (not necessarily connected).
 Then $b(\Gamma)=b(\Gamma')+b(\Gamma/\Gamma').$
\end{lem}

If $\Gamma$ is a connected non-contractible graph, the \emph{core} of
$\Gamma$, denoted by $Core(\Gamma)$, is the smallest connected subgraph $\Gamma_0$ of $\Gamma$
such that the inclusion map $\iota: \Gamma_0\to\Gamma$ is a homotopy
equivalence. A connected graph $\Gamma$ is called a \emph{core graph}
if $\Gamma=Core(\Gamma)$. Note that if $\Gamma$ is a finite connected
non-contractible graph then $\Gamma$ is a core graph if and only if for
every $v\in V\Gamma$ we have $\deg(v)\ge 2$. If $\Gamma$  is a graph and $v\in V\Gamma$ a vertex then we call the pair $(\Gamma,v)$ a \emph{core pair} or a \emph{core graph with respect to $v$} if $\Gamma$ contains no proper subgraph containing $v$ such that the inclusion map is a homotopy equivalence. Thus if $(\Gamma,v)$ is a core graph then either $\Gamma$ is a core graph or $\Gamma$ is obtained from a core graph $\bar\Gamma$ by adding segment joining $v$ and $\bar\Gamma$.

Note that if $f:\Delta\to\Gamma$ is a morphism of graphs and if
$\beta=e_1,\dots, e_n$ is an edge-path in $\Delta$ of length $n>0$ then
$f(\beta):=f(e_1),\dots f(e_n)$ is an edge-path in $\Gamma$ with
$|\gamma|=|\beta|$. In this case we call $\beta$ a \emph{lift} of
$f(\beta)$ to $\Delta$. Also, if $u\in V\Delta$ and $v=f(u)\in V\Gamma$
then we call the length-$0$ path $\beta=u$ a \emph{lift} of the
length-$0$ path $\gamma=v$.

It is not hard to see that a graph morphism $f:\Delta\to\Gamma$ is a
covering map (in the standard topological sense) if and only if $f$ is
surjective and locally bijective.
It follows from the definitions that a surjective morphism
$f:\Delta\to\Gamma$ is a covering map if and only if for every vertex
$v\in V\Gamma$, every edge-path $\gamma$ in $\Gamma$ with
$\alpha(\gamma)=v$ and every vertex $u\in V\Delta$ with $f(u)=v$ there
exists a unique lift $\widetilde\gamma$ of $\gamma$ such that
$\alpha(\widetilde\gamma)=u$.

We will need to investigate the following weaker version of being a
covering map:

\begin{defn}[Path-surjective map]
Let $f:\Delta\to\Gamma$ be a surjective morphism of graphs. We say
that $f$ is \emph{path-surjective} if for every reduced edge-path
$\gamma$ in $\Gamma$ there exists a lift $\widetilde\gamma$ of $\gamma$ in $\Delta$.
\end{defn}

Let $\Delta$ be a finite  graph which contains at least one edge. A subgraph $\Lambda$ of $\Delta$ is an \emph{arc} if $\Delta$ is the image of a simple (possibly closed) edge-path $\gamma=e_1,\dots, e_n$, where $n\ge 1$, such that for all $1\le i <n$ the vertex $\omega(e_i)$ has degree $2$ in $\Delta$.  If, in addition, $\deg_\Delta(\alpha(e_1))\ne 2$ and $\deg_\Delta(\omega(e_n))\ne 2$, then $\Lambda$ is called a \emph{maximal arc} in $\Delta$. An arc $\Lambda$ in $\Delta$ is \emph{semi-maximal} if for the simple path  $\gamma=e_1,\dots, e_n$ defining $\Lambda$ either $\deg_\Delta(\alpha(e_1))\ne 2, \deg_\Delta(\omega(e_n))=2$ or $\deg_\Delta(\alpha(e_1))=2, \deg_\Delta(\omega(e_n))\ne 2$. 

If $\Delta$ is a finite connected graph with at least one edge and such that $\Delta$ is not a simplicially subdivided circle then every edge of $\Delta$ is contained in a unique maximal arc of $\Delta$. Moreover in this case for every  arc $\Lambda$ of $\Delta$ the simple edge-path $\gamma=e_1,\dots, e_n$, whose image is equal to $\Lambda$, is unique up to inversion and the set of its endpoints $\{\alpha(e_1), \omega(e_n)\}$ is uniquely determined by $\Lambda$.

\section{A characterization of 2-sheeted covers of a rose}\label{charrose}

\begin{defn}[Standard $n$-rose]
\smallskip For an integer $n\ge 1$ we denote by $R_n$ the rose with $n$ petals, i.e. the graph with a single vertex $v_0$ and edge set $ER_n=\{e_1^{\pm 1},\ldots ,e_n^{\pm 1}\}$. For $0\le m<n$ we consider $R_m$ as a subgraph of $R_n$ in the obvious way. 
\end{defn}

In this paper we will utilize some of the basic results about Stallings folds. We refer the reader to \cite{St,KM} for the relevant background information.

It is easy to explicitly characterize all the connected 2-fold covers of $R_n$.
Namely, a morphism $p:\Gamma\to R_n$ (where $\Gamma$ is a connected
graph) is a 2-fold cover if and only if $\Gamma$ has two distinct
vertices $v$ and $w$ and  for each $i=1,\ldots ,n$ there either exist loop edges at $v$ and $w$ that map to $e_i$ or there exist two edges from $v$ to $w$ that map to $e_i$ and $e_i^{-1}$, respectively. The connectedness of $\Gamma$ implies that for at least one $i$ the second option occurs.

\begin{figure}[h]
	\centering
	\input{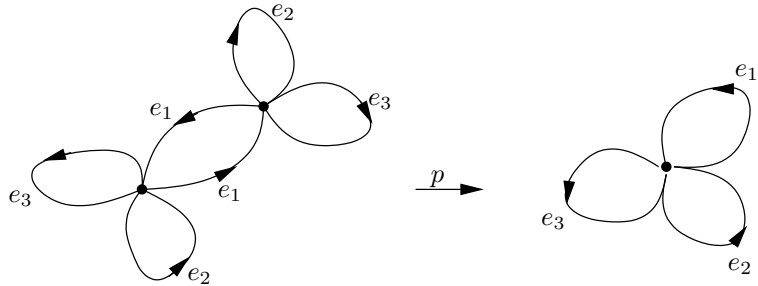}
	\caption{A 2-sheeted cover of $R_3$}
	\label{fig2sheeted}
\end{figure}

Note that if $p:\Gamma\to R_n$ is a morphism of graphs then
$p(V\Gamma)=\{v_0\}$ since $R_n$ is a graph with a single vertex
$v_0$. Thus any such morphism $p:\Gamma\to R_n$ can be uniquely
encoded by specifying for every edge $e\in E\Gamma$ its ``label''
$p(e)\in ER_n=\{e_1^{\pm 1}, \dots, e_n^{\pm 1}\}$. Moreover, a
labelling of $E\Gamma$ by elements of $ER_n$ determines a morphism if
and only if this labeling respects edge inversion. 

\smallskip The main purpose of this section is to characterize 2-fold covers of
$R_n$ in terms of path-surjective maps:

\begin{thm}\label{notallpathslift} Let $\Gamma$ be a connected finite core graph, $n\ge 2$ and $p:\Gamma\to R_n$ be a path surjective morphism such that the following hold:
\begin{enumerate}
\item $R_n$ does not lift to $\Gamma$, i.e. there exists no morphism $f:R_n\to\Gamma$ such that $p\circ f=id_{R_n}$.
\item $b(\Gamma)\le 2n-1$.
\end{enumerate}
Then $p$ is a 2-sheeted covering map.
\end{thm}

Note that it is easy to find a finite connected core graph $\Gamma$ and a path-surjective map $p:\Gamma\to R_n$ such that $\Gamma$ does not contain a finite-sheeted cover of $R_n$. Indeed taking $\Gamma$ to be rose with $(2n)(2n-1)$ loops of length $2$ and choosing $f$ to be the map that maps the $(2n)(2n-1)$ loops of $\Gamma$ to the $(2n)(2n-1)$ distinct reduced paths of length $2$ in $R_n$ does the trick. 
Thus the hypothesis on the Betti number in Theorem~\ref{notallpathslift} cannot be dropped. However, we do not know whether $2n-1$ is the best possible lower bound.

\begin{conv}
For the remainder of this section we assume that $n\ge 2$ and  $p:\Gamma\to R_n$ is as in Theorem~\ref{notallpathslift}.
For any nonempty subset $S\subset ER_n$ let $\Gamma_S$ be the subgraph of $\Gamma$ formed by the edges that get mapped by $p$ to some edge in $S\cup S^{-1}$. For $S=\{e\}$ we will write $\Gamma_e$ instead of $\Gamma_{\{e\}}$.  Note that for every $e\in ER_n$ we have $\Gamma_e=\Gamma_{e^{-1}}=\Gamma_{\{e, e^{-1}\}}$.
\end{conv}

\begin{lem}\label{lem:b}
The following hold:
\begin{enumerate}
\item For every $e\in \{e_1,\dots, e_n\}$ we have $b(\Gamma_e)\ge 1$.
\item There exists $e\in \{e_1,\dots, e_n\}$ such that $b(\Gamma_e)=1$.
\end{enumerate}

\end{lem}
\begin{proof}
To see that (1) holds let $e\in ER_n$ be arbitrary. Put $k=|E\Gamma|+1$ and consider the path $\gamma:=e^k=\underbrace{e,\dots ,e}_{k \text{ times}}$ in $R_n$. Since $p$ is  path-surjective,  there exists a lift $\widetilde\gamma$ of $\gamma$ to $\Gamma$, and, by definition of $\Gamma_e$, this path $\widetilde\gamma$ is contained in $\Gamma_e$. By the choice of $k$ the path  $\widetilde\gamma$ contains at least two occurrences of the same oriented edge, and hence $\widetilde\gamma$ contains a simple circuit embedded in $\Gamma_e$. Therefore $b(\Gamma_e)\ge 1$, as claimed. Thus part (1) of the lemma is established.

Suppose now that part (2) fails. By part (1)  this means that $b(\Gamma_e)\ge 2$ for every $e\in \{e_1,\dots, e_n\}$. Since $\Gamma=\cup_{i=1}^n \Gamma_{e_i}$, this implies that $b(\Gamma)\ge 2n$, contrary to the assumptions about $\Gamma$ in Theorem~\ref{notallpathslift}. Thus part (2) of the lemma also holds.
\end{proof}

In view of Lemma~\ref{lem:b}, after possibly permuting the indexing of the edges of $R_n$ we may assume that $b(\Gamma_{e_1})=1$.

\begin{lem}
There exists a permutation of $\{e_1,\dots, e_n\}$ such that, after this permutation, for all $k\in\{1,\ldots,n\}$ we have $$b(\Gamma_{S_k})\le 2k-1$$ where $S_k=\{e_1,\ldots ,e_k\}$.
\end{lem}

\begin{proof}

The proof is by induction on $k$. As noted above, in view of Lemma~\ref{lem:b} we may assume that $b(\Gamma_{e_1})=1$.

Suppose now that $2\le k\le n$ and that we have already re-indexed the edges of $R_n$ so that $b(\Gamma_{S_i})\le 2i-1$ for $i=1,\dots, k-1$.
In particular,  $b(\Gamma_{S_{k-1}})\le 2k-3$.  For $1\le i<j\le n$ put $S_i^j:=\{1,\ldots ,i\}\cup\{j\}$.

 Choose $j\in \{k,\dots, n\}$ so that $b(\Gamma_{S_{k-1}^j})=\min \{ b(\Gamma_{S_{k-1}^i})  | k\le i\le n\}$.
Then re-index the edges $e_k,\dots, e_n$ so that $j=k$. Thus now $b(\Gamma_{S_k})=b(\Gamma_{S_{k-1}^k})\le b(\Gamma_{S_{k-1}^i})$ for $i=k,\dots, n$.
 
By Lemma~\ref{lem:quotient},  we have $$b(\Gamma_{S_k})-b(\Gamma_{S_{k-1}})=b(\Gamma_{S_k}/\Gamma_{S_{k-1}}). $$ 

If $b(\Gamma_{S_k}/\Gamma_{S_{k-1}})\le 2$ then $b(S_k)\le (2k-3)+2=2k-1$ and $S_k$ has the properties required by the lemma, completing the inductive step.

Thus suppose that $b(\Gamma_{S_k}/\Gamma_{S_{k-1}})\ge 3$. The minimality assumption in the choice of $e_k$ implies that for all $i=k,\dots, n$ we have $b(\Gamma_{S_{k-1}^i}/\Gamma_{S_{k-1}})\ge 3$. Thus for $i=k,\dots, n$ the graph $\Gamma/\Gamma_{S_{k-1}}$ contains a graph of Betti number at least $3$ all of whose edges map to $e_i^{\pm 1}$ in $R_n$. Therefore $$ b(\Gamma_{S_k}/\Gamma_{S_{k-1}})\le b(\Gamma/\Gamma_{S_{k-1}})-3(n-k).$$
Recall that by assumptions on $\Gamma$ in Theorem~\ref{notallpathslift} we have $b(\Gamma)\le 2n-1$. 
Hence 
\begin{gather*}
b(\Gamma_{S_k})= b(\Gamma_{S_k}/\Gamma_{S_{k-1}})+ b(\Gamma_{S_{k-1}})
\le b(\Gamma/\Gamma_{S_{k-1}})-3(n-k)+ b(\Gamma_{S_{k-1}})= \\
=b(\Gamma)-3(n-k)\le (2n-1)-3(n-k)\le 2k-1,
\end{gather*}
as required. This completed the inductive step and the proof of the lemma.
\end{proof}

\begin{conv}
For the remainder of the section we assume that the edges of $R_n$ are indexed so that $b(\Gamma_{e_1})=1$ and that for every $k=1,\dots, n$ we have $1\le b(\Gamma_{S_k})\le 2k-1$.  Denote by $C_{e_1}$ the unique embedded circuit in $\Gamma_{e_1}$ and let $\Gamma_{e_1}^0$ be the component of $\Gamma_{e_1}$ containing $C_{e_1}$.

For $k=1,\dots, n$ let $\Delta_k$ be  the core of the component of $\Gamma_{S_k}$ containing $C_{e_1}$.
Thus $b(\Delta_k)\le b(\Gamma_{S_k})\le 2k-1$.

\end{conv}

\begin{lem}\label{lem:lifting}
For every $k=1,\dots, n$ every reduced path $\gamma$ in $R_k$ lifts to a path in $\Delta_k$.
\end{lem}
\begin{proof}
Let $1\le k\le n$ and let $\gamma$ be any reduced path in $R_k$. Choose $\epsilon,\delta\in \{-1, 1\}$ so that the path $e_1^\epsilon \gamma e_1^\delta$ is reduced.
Put $m=|E\Gamma|+1$.  Thus $\gamma_m=e_1^{m\epsilon} \gamma e_1^{m\delta}$ is a reduced path in $R_k$. Since $p:\Gamma\to R_n$ is path-surjective, the path $\gamma_m$ lifts to some path $\widetilde{\gamma}_m$ in $\Gamma$. The choice of $m$ implies that $\widetilde{\gamma}_m$ is contained in the component of $\Gamma_{S_k}$ containing $C_{e_1}$, that is, the component $Y$ of $\Gamma_{S_k}$ containing $\Delta_k$. Note that $\Delta_k=core(Y)$ and $C_{e_1}\subseteq \Delta_k$.
Each of the two subpaths of $\widetilde{\gamma}_m$ corresponding to the lifts of $e_1^{m\epsilon}$ and  $e_1^{m\delta}$ overlaps $C_{e_1}$ and hence overlaps $core(Y)$. 
Therefore the portion of $\widetilde{\gamma}_m$  corresponding to the lift of $\gamma$ is contained in $core(Y)=\Delta_k$. Thus $\gamma$ lifts to $\Delta_k$, as claimed.
\end{proof}

We now examine the properties of the graph $\Delta_2$.

\begin{prop}\label{prop:k=2}
One of the following holds:
\begin{enumerate}
\item The restriction of $p$ to $\Delta_2$ is a 2-sheeted cover onto $R_2$.
\item $R_2$ lifts to $\Delta_2$. 
\end{enumerate} 
\end{prop}

\begin{proof}

Recall that  $C_{e_1}$ is the unique embedded circuit in $\Gamma_{e_1}$ and that $\Gamma_{e_1}^0$ is the component of $\Gamma_{e_1}$ containing $C_{e_1}$.
Thus $C_{e_1}$ is labelled by a power of $e_1$.

 Note that there must exist a circuit $C_{e_2}$ labeled by a power of $e_2$ that lies in a component $\Gamma_{e_2}^0$ of $\Gamma_{e_2}$ that intersects $\Gamma_{e_1}^0$ in some (not necessarily unique) vertex $v$ as any path of type $e_1^ ne_2^ n$ is assumed to lift to $\Gamma_{S_2}\subset \Gamma$. We distinguish a number of subcases, which cover all possibilities:

\medskip \noindent{\bf Subcase A: } $b(\Gamma_{e_2})=b(\Gamma_{e_2}^0)=2$. In this case we have $\Delta_2\subset\Gamma_{e_1}^ 0\cup \Gamma_{e_2}^0$ and $\Gamma_{e_1}^ 0\cap \Gamma_{e_2}^0=\{v\}$ as $\Delta_2$ is a core graph with $b(\Delta_2)\le 3$. It follows that $R_2$ lifts to $\Delta_2$  as by Lemma~\ref{lem:lifting} the path $e_1,e_2,e_1,e_2$ must have a lift which implies that for $i=1,2$ there must be a loop edge based at $v$ that is mapped to $e_i$.

\medskip \noindent{\bf Subcase B: } $b(\Gamma_{e_2})=2$ and $b(\Gamma_{e_2}^0)=1$. Again we must have $\Gamma_{e_1}^ 0\cap \Gamma_{e_2}^0=\{v\}$ for Betti number reasons. Moreover there exists a component $\Gamma_{e_2}^1$ of $\Gamma_{e_2}$ with $\Gamma_{e_2}^1\neq\Gamma_{e_2}^0$ and  $b(\Gamma_{e_2}^ 1)=1$. 

\smallskip If $\Gamma_{e_1}^0\cap \Gamma_{e_2}^1=\emptyset$ then any path of type $e_1^ n,e_2^n,e_1^ n,e_2^n$ with $n\ge |E\Gamma|$ must lift to $\Gamma_{e_1}^0\cup \Gamma_{e_2}^0$ as lifts are connected and any subpath $e_i^n$ must lift to a component of $\Gamma_{e_i}$ that contains a non-trivial circuit. Thus any path $e_i^n$ must lift to a closed path in $\Gamma_{e_i}^0$ based at $v$ for all $n\ge |E\Gamma|$.  As $b(\Gamma_{e_1}^0)=b(\Gamma_{e_2}^0)=1$ this  implies that $\Gamma_{ e_i}^0$ contains a loop edge based at $v$ that is mapped to $e_i$ for $i=1,2$, thus $R_2$ lifts to $\Delta_2$.

\smallskip If $\Gamma_{e_1}^0\cap \Gamma_{e_2}^1\neq \emptyset$ then  $\Gamma_{e_1}^0\cap \Gamma_{e_2}^1$ must consist of a single vertex $w$ for Betti number reasons. Moreover we have $\Delta_2\subset \Gamma_{e_1}^ 0\cup \Gamma_{e_2}^ 0\cup \Gamma_{e_2}^ 1$. By Lemma~\ref{lem:lifting} the path $e_1,e_2,e_1,e_2,e_1,e_2,e_1$ lifts to $\Delta_2$ and the lifts of the occurrences of $e_2$ must be loop edges based at $v$ or $w$. If the lifts of two consecutive occurrence of $e_2$ are identical then there must also be a loop edge mapped to $e_1$ at the same base vertex and it follows that $R_2$ lifts. Otherwise there must be an edge from $v$ to $w$ and also an edge from $w$ to $v$ that is mapped to $e_1$ in which case the restriction of $p$ to $\Delta_2$ is a $2$-sheeted cover onto $R_2$.

\medskip \noindent{\bf Subcase C: } $b(\Gamma_{e_2})=b(\Gamma_{e_2}^0)=1$ and $\Gamma_{e_1}^0\cap \Gamma_{e_2}^0=\{v\}$. In this case any path of type $e_1^ n,e_2^n,e_1^ n,e_2^n$ with $n\ge |E\Gamma|$ must lift to $\Gamma_{e_1}^0\cup \Gamma_{e_2}^0$ and we argue as in Subcase B that $R_2$ lifts.

\medskip \noindent{\bf Subcase D: } $b(\Gamma_{e_2})=b(\Gamma_{e_2}^0)=1$ and $\Gamma_{e_1}^0\cap \Gamma_{e_2}^0=\{v,w\}$. 

If $\Delta_2$ contains no loop edge at $v$ or $w$ then there must be edges from $v$ to $w$ that map to $e_1$, $e_2$, $e_1^{-1}$ and $e_2^ {-1}$  as, by Lemma~\ref{lem:lifting},   the path $e_1,e_2,e_1^ {-1}, e_2^ {-1},e_1,e_2$ lifts to $\Delta_2$, and the lift must alternate between $v$ and $w$. It follows that the restriction of $p$ to $\Delta_2$ is a 2-sheeted cover of $R_2$.

In the remaining case we assume by symmetry that $\Gamma_{e_1}^0$ has a loop edge based at $v$. We may assume that $\Gamma_{e_2}^ 0$ does not have a loop edge based at $v$, otherwise $R_2$ lifts to $\Delta_2$. Let now $f$ be the unique edge of  $\Gamma_{e_1}^ 0\cap\Delta_2$ with $\alpha(f)=w$. Choose $\varepsilon\in\{-1,1\}$ such that $p(f)=e_1^{\varepsilon}$. It follows that any lift of $e_1^{\varepsilon n}$ for $n\ge |E\Gamma|$  to $\Delta_2$ terminates  at $v$. As $e_1^{\varepsilon n}e_2e_1$ 
lifts to $\Delta_2$ and as there is no loop edge in $\Gamma_{e_2}^0$ based at $v$ it follows that any lift of $e_1^{\varepsilon n}e_2$ to $\Delta_2$ 
must terminate at $w$. It follows that the path $e_1^{\varepsilon n}e_2e_1^{-\varepsilon}$ does not lift to $\Delta_2$, thus this case does not occur.
\end{proof}

\begin{proof}[Proof of Theorem~\ref{notallpathslift}]

We claim that for $2\le k\le n$ one of the following holds:
\begin{enumerate}
\item The restriction of $p$ to $\Delta_k$ is a 2-sheeted cover onto $R_k$.
\item $R_k$ lifts to $\Delta_k$.
\end{enumerate} 

For $k=n$ the above claim implies the statement of Theorem~\ref{notallpathslift} since $\Gamma=core(\Gamma)=\Delta_n$.
Note that the claim above is false for $k=1$.

We establish the claim by induction on $k$. The base case, $k=2$, is exactly the conclusion of Proposition~\ref{prop:k=2}.

Suppose now that $3\le k\le n$ and that the claim has been verified for $k-1$.  Thus  we know that $\Delta_{k-1}$ is either a 2-sheeted cover of $R_{k-1}$ or that $\Delta_{k-1}$ contains a lift of $R_{k-1}$. In the first case we may moreover assume that there is an edge with label $e_1^{-1}$ and an edge with label $e_1$ connecting the two vertices of $\Delta_{k-1}$.

\medskip \noindent{\bf Subcase A: } $\Delta_{k-1}$ contains a lift $\tilde R_{k-1}$ of $R_{k-1}$. If $\Delta_k$ contains a loop edge at the base vertex of $\tilde R_{k-1}$ that is mapped to $e_k$ then  $\Delta_k$ contains a lift of $R_k$ and there is nothing to show. Thus we may assume that such a loop edge does not exist. We will show that this yields a contradiction.

Let $\Gamma_{1k}$ be the core of the component of the subgraph $\Gamma_{\{e_1,e_k\}}$ containing the unique loop edge mapped to $e_1$. Moreover let $R^{1k}$ be the subgraph  of $R_n$ consisting of the edge pairs $e_1^{\pm 1}$ and $e_k^{\pm 1}$.  By an argument similar to the proof of Lemma~\ref{lem:lifting}, any path $\gamma$ in $R^{1k}$  must lift to $\Gamma_{1k}$.  It now follows from the case $k=2$ and the fact that  $R^{1k}$  does not lift to $\Gamma$ that $\Gamma_{1k}$ is either a 2-sheeted cover of $R^{1k}$  or that $b( \Gamma_{1k})\ge 4$. As $\Gamma_{1k}$ does not contain a second loop edge mapped to $e_1$ it cannot be a 2-sheeted cover,  thus $b( \Gamma_{1k})\ge 4$. Let now $\Gamma':=\Gamma_{1k}\cup \tilde R_{k-1} $. It follows that $$b(\Gamma')\ge 4+ (k-2)=k+2.$$

\smallskip\noindent{\bf Claim: } For $2\le i\le k-1$  either $\Delta_k$ contains a circuit that maps to $e_i$ and that is distinct from the loop in $\tilde R_{k-1}$ or there exist two edges in $E\Delta_k-E\Gamma'$ originating at $\Gamma'$ that map to $e_i$ and $e_i^{-1}$, respectively.

\smallskip\noindent{\em Proof of Claim: }Suppose that no such circuit exists for some $i\in\{2,\ldots k-1\}$. We need to show that for $\varepsilon\in\{-1,1\}$ there exists an edge in $E\Delta_k-E\Gamma'$ originating at $\Gamma'$ that maps to $e_i^{\varepsilon}$. Suppose that  $\varepsilon\in\{-1,1\}$ and that no such edge exists. 
It follows that any lift of the path $e_i^{\varepsilon\cdot |E\Gamma|}$ must terminate at the base vertex $v_0$ of $\tilde R_{k-1}$ and any lift of $e_i^{\varepsilon\cdot |E\Gamma|}e_k$ must terminate at a vertex distinct from $v_0$. This implies  that the path $e_i^{\varepsilon\cdot |E\Gamma|}e_ke_i^{\varepsilon}$ does not lift to $\Delta_k$, a contradiction. The claim follows.\hfill$\Box$

\smallskip Now each additional circuit gives a contribution of at least 1 to the Betti number and two edges originating at $\Gamma'$ do the same as we are looking at a core graph, thus the loose ends must close up. It follows that $$b(\Gamma_k)\ge b(\Gamma')+(k-2)=(k+2)+(k-2)=2k$$ which gives the desired contradiction.

\medskip \noindent{\bf Subcase B: } $\Delta_{k-1}$ is a 2-sheeted cover of $R_{k-1}$ and the edges with label $e_1$ span a circuit of length $2$. It follows in particular that $b(\Gamma_{k-1})=2k-3$.

Define $\Gamma_{1k}$ as in Subcase A, again we apply the case $k=2$ to this graph. As $\Gamma$ contains no loop edge with label $e_1$ and as for Betti number reasons $b(\Gamma_{1k})\le 3$ it follows that $\Gamma_{1k}$ is a 2-sheeted cover of $R_{1k}$. It follows that $\Gamma_{1k}$ consists either of a circuit of length two with label $e_1,e_1$ and two loop edges with label $e_k$ or of two circuits of length $2$ with label $e_1,e_1$ and $e_k,e_k$ with common vertices. In both cases it is immediate that $\Delta_{k-1}\cup \Gamma_{1k}$ is a 2-sheeted cover of $R_k$.

\end{proof}

\section{Lifting random paths in $R_n$}\label{liftingrandompaths}

In this section all edges are geometric edges, i.e. edge pairs. Theorem~\ref{notallpathslift} is only used in this section, specifically in the proof of Theorem~\ref{readingranndomwordsingraphs}.

\begin{defn}
Let $\alpha\in [0,1]$. We say that a path $\gamma$ in some graph $\Gamma$ is $\alpha$-injective if $\gamma$ crosses at least $\alpha\cdot|\gamma|$ distinct topological edges. Thus $\gamma$ is $1$-injective if and only if $\gamma$ travels no topological edge twice. In other words, $\gamma$ is $\alpha$-injective if its image is of volume at least $\alpha\cdot |\gamma|$.
\end{defn}

\begin{defn}\label{def:q-injective}
Let $\Gamma$ be a graph and let $\gamma=e_1,\dots,e_k$ be an edge-path in $\Gamma$. We say that the edge $e_i$ is \emph{$\gamma$-injective} if the topological edge of $\Gamma$ corresponding to $e_i$ is traversed by $\gamma$ exactly once. Similarly, a subpath $\gamma'=e_q\dots e_r$  (where $1\le q\le r\le k$) of $\gamma$ is \emph{$\gamma$-injective} if for every $i=q,\dots, r$ the edge $e_i$ is \emph{$\gamma$-injective}.

Similarly, let $\Upsilon$ be a graph which is a topological segment subdivided into finitely many edges, let $\Gamma$ be a graph and let $f:\Upsilon\to \Gamma$ be a graph map. 
We say that an edge $e$ of $\Upsilon$ is \emph{$\Upsilon$-injective} if for the edge-path $\gamma$ determined by $\Upsilon$ the edge $e$ is $\gamma$-injective.  An arc of $\Upsilon$ is \emph{$\Upsilon$-injective} if every edge of this arc is $\Upsilon$-injective. 
\end{defn}

\medskip In the following we denote the set of all reduced paths in
$R_n$ by $\Omega$ and the set of all reduced paths of length $N$ by
$\Omega_N$. We further call a subset $S\subset \Omega$ \emph{generic}
if $\underset{N\to\infty}{\lim}\frac{|S\cap
  \Omega_N|}{|\Omega_N|}=1$. This definition of genericity agrees with
the more detailed notions of genericity defined in
Section~\ref{S:generic} below.

 Throughout this section we assume that $n\ge 2$.

\medskip The purpose of this section is to establish the following theorem:

\begin{thm}\label{readingranndomwordsingraphs} Let $\alpha\in[0,1)$. The set $\Omega$ of all reduced paths in $R_n$ contains a generic subset $S$ such that the following holds:

Let $s\in S$ and $(\Gamma,v_0)$ be a connected core graph with  $b(\Gamma)\le 2n-1$. Suppose that $f:\Gamma\to R_n$ is a morphism such that $\Gamma$ does not contain a finite subgraph $\Gamma'$ such that $p|_{\Gamma'}:\Gamma'\to R_n$ is a covering (of degree $1$ or $2$). Then any lift $\tilde s$ of $s$ is $\alpha$-injective.
\end{thm}

Before we give the proof Theorem~\ref{readingranndomwordsingraphs} we establish Proposition~\ref{nonperiodic} and Proposition~\ref{manysubpaths} which follow from elementary probabilistic considerations. 

\begin{prop}\label{nonperiodic} For any $\varepsilon>0$ there exists a constant $L=L(\varepsilon,n)$ and a generic subset $S\subset \Omega$ such that the following holds:

Suppose that $s\in S$ and that $\gamma$ is a path in $R_n$ of length at least $L$ such that $\gamma$ occurs as least twice as a subpath of $s$ such that these two subpaths do not overlap.  Then the total length of any collection of non-overlapping subpaths of $s$ that coincide with $\gamma$ is bounded from above by $\varepsilon|s|$.
\end{prop}

For the remainder of this section, unless specified otherwise,  we fix the following conventions and notations.
\begin{conv} $\, $

\smallskip (1) Denote by $\Omega_N$ the set of of all reduced paths of length $N$, endowed with the uniform distribution. Let $\gamma$ be a fixed path in $R_n$ and put $k:=|\gamma|$. Note that in $R_n$ there are  exactly $2n(2n-1)^{k-1}$ reduced paths of length $k$.

Now consider the random variable $Y=Y_{\gamma, N}:\Omega_N\to\mathbb R$ defined by $$Y_{\gamma, N}(s):=|\gamma|\cdot \max\{l\mid s\hbox{ has }l\hbox{ disjoint  subpaths that coincide with }\gamma\}.$$ Thus $Y$ assigns to any path $s\in \Omega_N$ the total length of a maximal portion of $s$ that is covered by a collection of disjoint copies of $\gamma$. To prove Proposition~\ref{nonperiodic} we will need to estimate $P(Y\ge \epsilon N)$ from above. Our estimate will not depend on $\gamma$ but only on $k=|\gamma|$.

\smallskip (2) Put $p:= \frac{1}{(2n-1)^{k}}$. Note that for $n\ge 2$ and $k\ge 1$ we have $0<2p<1$.
We will also consider a random variable $X=X_k$ that is binomially distributed with parameters $N$ and $2p$.

In addition, let $\Omega_N'=\{0,1\}^N$ with the distribution corresponding to performing $N$ independent tosses of a coin such that for a single toss the probability of the coin landing "heads" is $2p$ and of it landing "tails" is $1-2p$. Thus for a binary string $w\in \Omega_N'$ we have $P(w)=(2p)^t(1-2p)^{N-t}$ where $t$ is the number of $1$'s in the string $w$. Hence for every integer $t\in [0,N]$  $P(X_k=t)$ is equal to the probability of the event that a randomly chosen string $w\in \Omega_N'$ contains exactly $t$  entries $1$.

\smallskip (3) Consider the finite state Markov chain $M$ generating freely reduced
words in $F_n=F(a_1,\dots, a_n)$. Thus the state set of $M$ is
$Q=\{a_1,\dots, a_n\}^{\pm 1}$ with transition probabilities
$P_M(a,b)=\frac{1}{2n-1}$ if $a,b\in Q, b\ne a^{-1}$ and
$P_M(a,b)=0$ if $a,b\in Q$ and $b= a^{-1}$. Note that $M$ is an irreducible finite state Markov chain with the stationary distribution being the uniform distribution on $Q$.  Since $n\ge 2$, whenever $a,b\in Q$ satisfy $P_M(a,b)>0$ then
\[ 
\frac{3}{4}\cdot  \frac{1}{2n} \le P_M(a,b) =\frac{1}{2n-1}\le  \frac{4}{3} \cdot  \frac{1}{2n}.\tag{$\clubsuit$}
\]
\end{conv}

\begin{lem}\label{Claim1}  There exists a constant $C_1=C_1(n,\epsilon)\ge 1$ such that if $k=|\gamma|\ge C_1$ then for every $N\ge 1$ and every integer $t\in [0,N]$ satisfying $t\ge \epsilon N/k$ we have
$(4/3)^t p^t \le (2p)^t (1-2p)^{N}$.
\end{lem}

\begin{proof} Let $n\ge 2$, $N\ge 1$, $k\ge 1$ be arbitrary integers and let $\epsilon>0$. Recall that $p=1/(2n-1)^k$ and that  $0<1-2p<1$.

Let $t\in [0,N]$ be an integer satisfying $t\ge \epsilon N/k$, clearly $t\ge 1$. As $t>0$ the  inequality $(4/3)^t p^t\le (2p)^t(1-2p)^N$ is equivalent to $(4/3)p\le 2p(1-2p)^{N/t}$ which in turn is equivalent to \begin{equation}2/3\le (1-2p)^{N/t}.\tag{$\ast$}\end{equation}

Thus we need to show that ($\ast$) holds provided $k$ is sufficiently large. Since $t\ge \epsilon N/k$, we have $N/t\le k/\epsilon$. As $0<1-2p<1$, it follows that $(1-2p)^{k/\epsilon} \le (1-2p)^{N/t}$.

 For fixed $n$ and $\epsilon$ we have  
 \[
 (1-2p)^{k/\epsilon}=\left(1-\frac{2}{(2n-1)^k}\right)^{k/\epsilon} \xrightarrow{k\to\infty} 1.
 \]
 Hence there exists a constant $C_1=C_1(n,\epsilon)\ge 1$ such that for every $k\ge C_1$ we have $(1-2p)^{k/\epsilon}\ge 2/3$. Then for any $k\ge C_1$
 \[
 (1-2p)^{N/t}\ge (1-2p)^{k/\epsilon} \ge 2/3,
 \]
as required. Thus the lemma is proven.
\end{proof}

\begin{lem}\label{Claim2} Let $C_1=C_1(n,\epsilon)\ge 1$ be the constant provided by Lemma~\ref{Claim1}. Suppose that $k=|\gamma|\ge C_1$.

Then for every integer $t\in [0,N]$ such that $t\ge \epsilon N/k$ we have \[P(Y_{\gamma,N}\ge  kt)\le P(X_k\ge t).\]
\end{lem}

\begin{proof} 

For any integers $1\le i_1<\dots <i_t\le N$ let $V(i_1,\dots,i_t)\subseteq \Omega_N$ be the event that for $s\in \Omega_N$ for each $j=1,\dots t$ the subpaths of $s$ of length $k$ starting 
in positions $i_1,\dots, i_t$ are equal to $\gamma$ and that these subpaths do not overlap in $s$.

We first claim that for every tuple $1\le i_1<\dots <i_t\le N$ we have \[P(V(i_1,\dots,i_t))\le (4/3)^t p^t. \tag{$\heartsuit$}\]

Note that the uniform distribution on $\Omega_N$ is the same as the distribution on $\Omega_N$ obtained by taking the uniform distribution on $Q$, considered as giving the initial letter of a word of length $N$, followed by applying the Markov chain $M$ exactly $N-1$ times.

Choose a specific $t$-tuple $1\le i_1<\dots <i_t\le N$.

If $N-i_t<k-1$, then the terminal segment of $s$ starting with the letter in position $i_t$ has length $<k$ and hence $P(V(i_1,\dots,i_t))=0\le (4/3)^t p^t$. Similarly, if there exists $j$  such that $i_{j}-i_{j-1}\le k-1$  then the subpaths of $s$ of length $k$ starting in positions $i_{j-1}$ and  $i_j$ overlap, and hence $P(V(i_1,\dots,i_t))=0\le (4/3)^t p^t$. Thus we assume that $N-i_t\ge k-1$ and that  $i_{j}-i_{j-1}\ge k$.  For similar reasons, we may assume that whenever $i_{j}-i_{j-1}=k$  then the first letter of $\gamma$ is not the inverse of the last letter of $\gamma$ since otherwise we again have $P(V(i_1,\dots,i_t))=0$.

By symmetry, the probability that the $i_1$-th letter of a random $s\in \Omega_N$ is the same as the first letter of $\gamma$ is $\frac{1}{2n}$. Then, applying the Markov chain $M$ defining above we see that the probability of reading $\gamma$ starting at letter number $i_1$ in $s$ is equal to $\frac{1}{2n(2n-1)^{k-1}}$, so that $P(V(i_1))=\frac{1}{2n(2n-1)^{k-1}}\le \frac{1}{(2n-1)^k}=p$. 

Given that the event $V(i_1)$ occurred, consider the conditional distribution $\nu_{1}$ on $Q$ corresponding to the letter in $s$ in the position $i_1+k$, i.e. the first letter in $s$ immediately after the last letter of the $\gamma$-subword that started in position $i_1$. By $(\clubsuit )$ it follows that for every $a\in Q$ we have $\nu_1(a)\le (4/3) \frac{1}{2n}$.  
For each $a\in Q$ let $H_{a, i_1+k}$ be the event that for an element of $\Omega_N$ the letter in position $i_1+k$ is $a$. Let $P(V(i_2)|H_{a,i_1+k})$ be the conditional probability of $V(i_2)$ given $H_{a,i_1+k}$. Then the conditional probability of $V(i_1,i_2)$ given that $V(i_1)$ occurred can be computed as $$P(V(i_1,i_2)| V(i_1))=\sum_{a\in Q} \nu_1(a) P(V(i_2)|H_{a,i_1+k}).$$
Also, the unconditional probability $P(V(i_2))$ can be computed as  $P(V(i_2))=\sum_{a\in Q} \frac{1}{2n} P(V(i_2)|H_{a,i_1+k})$. The same argument as the argument above for computing $P(V(i_1))$ shows that $P(V(i_2))=\frac{1}{2n(2n-1)^{k-1}}\le \frac{1}{(2n-1)^k}=p$.  Since for every $a\in Q$ we have $\nu_1(a)\le (4/3) \frac{1}{2n}$, it follows that 
\[
P(V(i_1,i_2)| V(i_1))= \sum_{a\in Q} \nu_1(a) P(V(i_2)|H_{a,i_1+k}) \le \] \[\le (4/3) \sum_{a\in Q} \frac{1}{2n} P(V(i_2)|H_{a,i_1+k}) = (4/3)P(V(i_2))\le (4/3)p.
\]

Similarly, $P(V(i_1,i_2,i_3)| V(i_1,i_2))\le (4/3)p$ and, so on, up to \[P(V(i_1,i_2,i_3,\dots, i_t)| V(i_1,i_2,\dots, i_{t-1}))\le (4/3)p.\]

Hence
\begin{gather*}
P(V(i_1,\dots,i_t))=\\
P(V(i_1))P(V(i_1,i_2)| V(i_1))\dots P(V(i_1,i_2,i_3,\dots, i_t)| V(i_1,i_2,\dots, i_{t-1}))\\ \le (4/3)^{t-1}p^t\le (4/3)^t p^t
\end{gather*}
as required. Thus $(\heartsuit)$ is verified.

For $1\le i_1<\dots <i_t\le N$ let $W(i_1,\dots, i_t)\subseteq \Omega_N'$ be the event that for a binary string $w\in \Omega_N'$ the digits in the positions $i_1,\dots, i_t$ are equal to $1$ and for all $j<i_t$, $j\ne i_1,\dots, i_t$ the digit in position $j$ in $w$ is $0$.  Then, by independence, $P(W(i_1,\dots, i_t))=(2p)^t(1-2p)^{i_t-t}$.  Hence, by Lemma~\ref{Claim1} and by $(\heartsuit)$, we have
\[
P(V(i_1,\dots,i_t))  \le (4/3)^tp^t \le  (2p)^t(1-2p)^{N} \le (2p)^t(1-2p)^{i_t-t} =P(W(i_1,\dots, i_t)).
\]

The event  "$Y_{\gamma,N}\ge  kt$" is the (non-disjoint) union of events $V(i_1,\dots,i_t)$ taken over all $t$-tuples $1\le i_1<\dots <i_t\le N$. Also, the event that $w\in \Omega_N'$ has at least $t$ digits $1$ is the \emph{disjoint} union of $W(i_1,\dots, i_t)$, again taken over all $t$-tuples $1\le i_1<\dots <i_t\le N$.
Therefore
\begin{gather*}
P(Y_{\gamma,N}\ge  kt)\le \sum_{1\le i_1<\dots <i_t\le N} P(V(i_1,\dots,i_t))\le \\ \sum_{1\le i_1<\dots <i_t\le N} P(W(i_1,\dots,i_t))=P(X_k\ge t),
\end{gather*}
as required. This completes the proof of Lemma~\ref{Claim2}.
\end{proof}

Having established Lemma~\ref{Claim1} and Lemma~\ref{Claim2}, we can now prove Proposition~\ref{nonperiodic}.

\begin{proof}[Proof of Proposition~\ref{nonperiodic}] Mean and variance for $kX_k$ are given by $$\mu(kX_k)=k\cdot\mu(X_k)=k\cdot N\cdot 2p=k\cdot N\cdot \frac{2}{(2n-1)^{k}}$$ and   $$var(kX_k)=k^2\cdot var(X_k)=k^2\cdot N\cdot 2p\cdot(1-2p)=$$ $$=k^2\cdot N\cdot \frac{2}{(2n-1)^{k}}\cdot (1-\frac{2}{(2n-1)^{k}})<k^2\cdot N\cdot \frac{2}{(2n-1)^{k}}.$$ 

Choose $L>C_1(n,\epsilon)$ such that $\mu(kX_k)\le \frac{\epsilon}{2}\cdot N$ for all $k\ge L$. It now follows from Lemma~\ref{Claim2} and the Chebyshev inequality (as $\epsilon N>\frac{\epsilon}{2} N\ge \mu(kX_k)$) that for any reduced path $\gamma$ with  $|\gamma|=k\ge L$ we have  $$P(Y_{\gamma,N}\ge \epsilon N)\le P(X_k\ge \frac{\epsilon N}{k})=P(kX_k\ge \epsilon N)\le \frac{var(kX_k)}{(\epsilon N-\mu(kX_k))^2}\le $$  $$\le \frac{var(kX_k)}{(\epsilon N-\frac{\epsilon}{2} N)^2}=\frac{4\cdot var(kX_k)}{(\epsilon N)^2}\le \frac{k^2\cdot N\cdot \frac{8}{(2n-1)^{k}}}{\epsilon^2N^2}=\frac{k^2}{\epsilon^2N}\cdot \frac{8}{(2n-1)^{k}}.$$

By Proposition 2.2 of \cite{M} we know that we can assume that no path $\gamma$ of length greater than $C_0\ln(N)$ with $C_0=\frac{11}{\ln(2n-1)}$ occurs twice in a random word of length $N$. Thus we only need to sum up the probabilities for all paths $\gamma$ of length up to $C_0\ln (N)$. Thus, taking into account that for each $k$ there are only $2n(2n-1)^{k-1}$ words to consider, the probability that a path of length at least $L$ and at most $C_0\ln (N)$  covers the portion $\epsilon\cdot N$ of a  path of $\Omega_N$ is bounded from above by $$\sum\limits_{k=L}^{C_0\ln(N)}2n(2n-1)^{k-1}\cdot \frac{k^2}{\epsilon^2N}\cdot \frac{8}{(2n-1)^{k}}=\frac{16n}{(2n-1)\epsilon^2N}\sum\limits_{k=L}^{C_0\ln(N)}k^2\le $$ $$\le \frac{16n}{(2n-1)\epsilon^2N}\cdot (C_0\ln(N))^3.$$

Now this number converges to $0$ as $N$ tends to infinity which proves the proposition.
\end{proof}

\begin{prop}\label{manysubpaths} Let $\alpha\in(0,1]$ and $\gamma$ be a reduced path in $R_n$. There exists a constant $\delta>0$ and a generic subset $S_{\gamma}\subset\Omega$ such that the following hold:

If $s\in S_{\gamma}$ can be written as a reduced product $s=s_0t_1s_1\cdot\ldots\cdot t_qs_q$ such that
\begin{enumerate}
\item $t_i$ does not contain $\gamma$ as a subpath for $1\le i\le q$ and 
\item $\sum\limits_{i=1}^q|t_i|\ge \alpha\cdot|s|$.
\end{enumerate}
Then $q\ge \delta\cdot |s|$
\end{prop}

\begin{proof}

The idea of the proof is simple: As an average subpath of a generic path that doesn't contain $\gamma$ as a subpath must be short, i.e. bounded in terms of $n$ and $\gamma$, and as the $t_i$ cover a definite portion of $s$ it follows that $q$ must be large is $|s|$ is large.

\smallskip We sketch a proof but omit some of the details for brevity. Suppose that $\gamma$ is a reduced path of length $k$ in $R_n$. Let $p=\frac{1}{2n(2n-1)^{k-1}}$. For any $j\in \mathbb N_{\ge 1}$ let  $f(j)$ be the probability that in a random semi-infinite reduced path in $R_n$, a subpath between two consecutive non-overlapping occurrences of $\gamma$ is of length~$j$.

 This defines a probability distribution $f$ on $\mathbb N_{\ge 1}$. This distribution is essentially geometric as the probability $f(l+k-1)$ is approximately $(1-p)^l$. Note that this calculation ignores the fact that the events that distinct subpaths are of type $\gamma$ are (slightly) correlated.  It follows in particular that the probability distribution $f$ has finite mean, i.e that the series $\sum\limits_{j=1}^\infty j\cdot f(j)$ converges. Put $\mu:=\sum\limits_{j=1}^\infty j\cdot f(j)$ and choose $j_0\in\mathbb N$ with $j_0>100k$ such that $$\sum\limits_{j=j_0}^\infty j\cdot f(j)\le \frac{\alpha}{10}\mu.$$

\smallskip Let now $s\in \Omega$. We consider the collection of subpaths that make up the complement of the union of all subpath of $s$ that are of type $\gamma$; we call these subpaths \emph{$\gamma$-complementary}. Then there exists a generic subset $\Omega'\subseteq \Omega$ such that for any $s\in\Omega'$ the collection of $\gamma$-complementary subsets is non-empty. For each $s\in\Omega'$ we obtain a probability distribution $f_s$ on $\mathbb N_{\ge 1}$ where $f_s(j)$ is  the number of $\gamma$-complementary subpaths of length $j$ in $s$ divided by the total number of $\gamma$-complementary subpaths $n_s$, i.e. $f_s$ is the relative frequency of $\gamma$-complementary subpaths of length $j$ in $s$. Then we have $$n_s\cdot \sum\limits_{j=1}^\infty j\cdot f_s(j)\le |s|$$ as the sum on the left is just the number of edges of $s$ that lie in the  $\gamma$-complementary subpaths of $s$.

\noindent{\bf Claim:} There exists some generic subset $S\subseteq \Omega'\subseteq \Omega$ such that for every $s\in S$ we have $$n_s\cdot \sum\limits_{j=j_0}^\infty j\cdot f_s(j)\le n_s\cdot \frac{\alpha}{5}\cdot \sum\limits_{j=1}^\infty j\cdot f_s(j).$$

To verify the claim note first that it can be shown using the law of large numbers, that there exists $\beta\in [0,1)$ such that for any $\epsilon>0$ there exists a generic set $S_\epsilon$ such that the following hold:

\begin{enumerate}
\item For any $s\in S_\epsilon$ $$(\beta-\epsilon) |s|\le n_s\cdot \sum\limits_{j=1}^\infty j\cdot f_s(j)\le (\beta+\epsilon) |s|.$$ 
\item For any $j\in\{1,\ldots j_0-1\}$ and $s\in S_\epsilon$ we have $$(\beta-\varepsilon)\cdot \frac{j\cdot f(j)}{\mu}\cdot |s|\le n_s\cdot j\cdot f_s(j)\le (\beta+\varepsilon)\cdot \frac{j\cdot f(j)}{\mu}\cdot |s|.$$ 
\end{enumerate}

 Indeed, the first assertion states for a generic  path  $s$ the $\gamma$-complementary paths of $s$ cover roughly a fixed portion $\beta$ of $s$, i.e. their total length is roughly $\beta|s|$. The second assertion states that out of this $\gamma$-complementary portion of total length $\beta|s|$,  the portion that consists of $\gamma$-complementary path of length $j$ is roughly $\frac{j\cdot f(j)}{\mu}$, i.e. that total length of the $\gamma$-complementary paths of length $j$ is roughly $\frac{j\cdot f(j)}{\mu}\cdot |s|$. This conclusion follows from the definition of $f$ and $\mu$.

\smallskip 
 Using (2) we get that for any $s\in S_\epsilon$ we have $$|s|\cdot \sum\limits_{j=1}^{j_0-1}(\beta-\varepsilon)\cdot \frac{j\cdot f(j)}{\mu}\le \sum\limits_{j=1}^{j_0-1}n_s\cdot j\cdot f_s(j)\le |s|\cdot \sum\limits_{j=1}^{j_0-1}(\beta+\varepsilon)\cdot \frac{j\cdot f(j)}{\mu}.$$

It follows that $$n_s\cdot \sum\limits_{j=j_0}^\infty j\cdot f_s(j)=n_s\cdot \sum\limits_{j=1}^\infty j\cdot f_s(j)-n_s\cdot \sum\limits_{j=1}^{j_0-1}  j\cdot f_s(j)\le $$ $$\le |s|\cdot (\beta+\epsilon)- |s|\cdot \frac{\beta-\epsilon}{\mu} \cdot \sum\limits_{j=1}^{j_0-1} j\cdot f(j)\le |s|\left(\beta+\epsilon-\frac{\beta-\epsilon}{\mu}(1-\frac{\alpha}{10})\cdot\mu\right)\le$$  $$\le |s|\cdot\left(\frac{\alpha\beta}{10}+ \epsilon\left(2-\frac{\alpha}{10}\right)\right).$$  For $\epsilon$ sufficiently small we further have $$|s|\cdot\left(\frac{\alpha\beta}{10}+\epsilon\left(2-\frac{\alpha}{10}\right)\right)\le \frac{\alpha}{5}(\beta-\epsilon)\cdot |s|\le \frac{\alpha}{5}\cdot n_s\cdot   \sum\limits_{j=1}^ \infty  j\cdot f_s(j)$$ which proves the claim.

Let now $s\in S$ and suppose that $s=s_0t_1s_1\cdot\ldots\cdot t_qs_q$ is as in the formulation of the lemma. Note that for $1\le i\le q$ either $|t_i|\le 2k-2$ or $t_i$ contains a subpath $t_i'$ of length at least $|t_i|-2k+2$ such that $t_i'$ is contained in a $\gamma$-complementary subpath of $s$. It follows from the claim that the total length of all $\gamma$-complementary subpaths of $s$ of length greater than $j_0$ is bounded from above by $$n_s\cdot \sum\limits_{j=j_0}^\infty j\cdot f_s(j)\le n_s\cdot \frac{\alpha}{5}\cdot \sum\limits_{j=1}^\infty j\cdot f_s(j)\le \frac{\alpha}{5}\cdot |s|.$$ Thus the sum of the lengths of those $t_i$ that contain a subpath of a $\gamma$-complementary subpath of length at least $j_0$ is bounded by $$\frac{j_0+2k}{j_0}\cdot \frac{\alpha}{5}\cdot |s|\le \frac{102k}{100k}\cdot \frac{\alpha}{5}\cdot |s|\le \frac{\alpha}{4}\cdot |s|.$$

It follows the sum of the length of those $t_i$ that are of length at most $j_0+2k$ must be at least $\frac{3\alpha}{4}\cdot |s|$ thus there must be at least $$\frac{1}{j_0+2k}\cdot \frac{3\alpha}{4}\cdot |s|=\frac{3\alpha}{4j_0+8k}\cdot |s|$$ such edges. Thus the conclusion of the lemma holds for $\delta:=\frac{3\alpha}{4j_0+8k}$.\end{proof}

We can now give the proof of the main result of this section.

\begin{proof}[Proof of Theorem~\ref{readingranndomwordsingraphs}] Put $\beta:=1-\alpha$. Recall that in Subsection~\ref{subsect:paths} we introduced the notion of arcs, maximal arcs and semi-maximal arcs in a finite graph. Note that it suffices to consider core pairs $(\Gamma,v_0)$ with $b(\Gamma)=2n-1$ as we can otherwise embed $(\Gamma,v_0)$ into a core pair $(\Gamma',v_0)$ with $b(\Gamma')=2n-1$ that still satisfies the hypothesis of Theorem~\ref{readingranndomwordsingraphs} and the conclusion for $(\Gamma',v_0)$ then implies the conclusion for $(\Gamma,v_0)$. 

Note that if $\Delta$ is a finite connected core graph with first Betti  number $m\ge 2$, then $\Delta$ has no degree-1 vertices and $\Delta$ is the union of at most $3m-3$ maximal arcs (which are obtained but cutting $\Delta$ open at all vertices of degree $\ge 3$). Also, if $(\Delta,v_0)$ is a finite connected core pair  with first Betti number $m\ge 2$, then $\Delta$ is the union of at most $3m-1$  maximal arcs.

Now let $(\Gamma,v_0)$ be as in Theorem~\ref{readingranndomwordsingraphs} with $b(\Gamma)=2n-1\ge 3$. As $(\Gamma,v_0)$ is a finite connected core graph the above remark implies that $\Gamma$ decomposes as the union of  at most $3(2n-1)-1=6n-4$ distinct maximal arcs. 
If $\gamma$ is a non-degenerate reduced edge-path in $\Gamma$ then after possibly subdividing the maximal arcs of $\Gamma$ containing endpoints of $\gamma$ along those endpoints,  we obtain a decomposition of $\Gamma$ as the union of distinct pairwise non-overlapping arcs $\Lambda_1,\dots \Lambda_t$ such that $t\le 6n-2<6n$ and such that $\gamma$ 
admits a unique decomposition as a concatenation of simple (and hence reduced) edge-paths
\[
\gamma=\gamma_0 \gamma_1\dots \gamma_r
\]
such that  for each $i=0,\dots, r$ there exists $y(i)\le t$ such that the image of  $\gamma_i$ is equal to the arc $\Lambda_{y(i)}$. 
Note that since the arcs  $\Lambda_1,\dots \Lambda_t$ are non-overlapping, for every $i,j\in \{0,\dots, r\}$ either $\gamma_i=\gamma_j^{\pm 1}$ or the paths $\gamma_i$ and $\gamma_i$ have no to topological edges in common.
We call the above factorization of $\gamma$ the \emph{canonical arc decomposition} of $\gamma$. We also call the arcs $\Lambda_1,\dots \Lambda_t$ the \emph{arc components of $\Gamma$ adapted to $\gamma$}.

{\bf Claim A.} We will show that  there exist constants $\ell_0,\ell_1,\ldots ,\ell_{2n}$ and generic sets $$S_{2n} \subseteq  S_{2n-1}\dots \subseteq S_1\subseteq S_0\subseteq \Omega$$   that only depend on $n$ and $\alpha$ such that the following hold:

\smallskip If $k\in\{0,1,\ldots ,2n\}$, $s\in S_k$, $(\Gamma,v_0)$ and  $f:\Gamma\to R_n$ as in the statement of the theorem with $b(\Gamma)=2n-1$ then for any non-$\alpha$-injective lift $\tilde s$ of $s$ there exist connected subgraphs $\Gamma_0\subset\Gamma_1\subset \ldots \subset \Gamma_{k}$ of $\Gamma$ with the following properties:
\begin{enumerate}
\item The volume of $\Gamma_i$ is bounded by above by $\ell_i$ for $0\le i\le k$.
\item  $b(\Gamma_i)\ge i$ for $0\le i\le k$.
\item If $\tilde s=e_1\dots e_N$ then $\#\{i | e_i\in E\Gamma_0\}\ge \frac{\beta\cdot N}{6n}=\frac{\beta\cdot |s|}{6n}$. 
\end{enumerate}

 \medskip Note that for $k=2n$ this claim implies the theorem. Indeed it implies that the generic set $S:=S_{2n}$ has the property that for any $s\in S$ any lift $\tilde s$ is $\alpha$-injective as a non-$\alpha$-injective lift would imply the existence of a subgraphs of $\Gamma$ of Betti number $2n$ which is impossible.
 
 \medskip The proof of Claim~A is by induction on $k$.

\medskip {\bf The case $k=0$.} Let $\varepsilon=\frac{\beta}{6n}$ and let $S$ and $L=L(\varepsilon,n)=L(\frac{\beta}{6n},n)$ be as in the conclusion of Proposition~\ref{nonperiodic}. We will show that the claim holds for $\ell_0:=L$ and $S_0:=S$. Let $s\in S_0$, $f:\Gamma\to R_n$ as in the statement of Theorem~\ref{readingranndomwordsingraphs}, and let $\tilde s$  a non-$\alpha$-injective lift of $s$. As $\tilde s$ is not $\alpha$-injective, it traverses fewer than $\alpha\cdot |s|$ distinct topological edges. This implies that there is a collection of subpaths of $\tilde s$ of total length at least $\beta\cdot |s|$ such that the image of these subpaths consists of all the topological edges visited at least twice by $\tilde s$.

\smallskip  Since $s$ is reduced and $\tilde s$ is a lift of $s$, the path $\tilde s$ is reduced as well. 
Let $\Lambda_1,\dots, \Lambda_t$, where $t\le 6n$, the the arc components of $\Gamma$ adapted to $\tilde s$ and consider the canonical arc decomposition of $\tilde s$.
Then the union of topological edges visited at least twice by $\tilde s$ is the union of some of the arcs $\Lambda_i$. Since $t\le 6n$, there exists $i_0\le t$ such that  the subpaths of $\tilde s$ that map to $\Lambda_{i_0}$ are of total length at least $\frac{\beta\cdot |s|}{6n}$ and that the number of such subpaths is at least $2$.  It follows from the conclusion of Proposition~\ref{nonperiodic} that $\Lambda_{i_0}$  has length at most $L$. Thus Claim~A follows by choosing  $\Gamma_0=\Lambda_{i_0}$ and putting $\ell_0=L$. Note that this choice of $\ell_0$ is independent of the map $f:\Gamma\to R_n$ and $s$.

\medskip {\bf The case $k\ge 1$.} By induction we already have the constants $\ell_0,\ldots ,\ell_{k-1}$ and a generic set $S_{k-1}$ such that for any $s\in S_{k-1}$ and any non-$\alpha$-injective lift $\tilde s$ of $s$ there exist subgraphs $\Gamma_0\subset\Gamma_1\subset\ldots\subset \Gamma_{k-1}$ with the properties specified in Claim~A. Note that since the volume of $\Gamma_{k-1}$ is $\le \ell_{k-1}$,  there are only finitely many possibilities for $\Gamma_{k-1}$ and for the map $f|_{\Gamma_{k-1}}:\Gamma_{k-1}\to R_n$.

\smallskip

{\bf Claim B.}  We will now show that for any pair $P:=(\Theta,p_{\Theta})$ consisting of a graph $\Theta$ and a morphism $p_{\Theta}:\Theta\to R_n$  
there exists a constant $\ell_{P}$ and a generic set $S_{P}\subset S_{k-1}$ such that the following hold: 

\smallskip Suppose that $f:\Gamma\to R_n$ is as above, $s\in S_{P}$, $\tilde s$ is a non-$\alpha$-injective lift. Let $\Gamma_{k-1}$ be the subgraph of $\Gamma$ whose existence is guaranteed by the induction hypothesis. If  $P=(\Gamma_{k-1},f|_{\Gamma_{k-1}})$ then there exists  a subgraph $\Gamma_k$ of $\Gamma$ containing $\Gamma_{k-1}$ such that $b(\Gamma_k)>b(\Gamma_{k-1})$ and that the volume of $\Gamma_k$ is bounded by $\ell_{P}$. 

\smallskip As there are only finitely many possibilities for $P$ and as the intersection of finitely many generic sets is generic, the conclusion of the inductive step for Claim~A follows from Claim B by taking $\ell_k$ to be the maximum of all occurring $\ell_{P}$ and taking $S_k$ to be the intersection of all $S_{P}$. 

\smallskip To establish Claim B we need to show the existence of $S_{P}$ and $\ell_{P}$ for a fixed pair $P$. Thus assume that $f:\Gamma\to R_n$ is as above, $s\in S_{k-1}$, $\tilde s$ is a non-$\alpha$-injective lift and  $\Gamma_{k-1}$ is as in the conclusion of the claim in the case $k-1$ for $s$ such that $P=(\Gamma_{k-1},f|_{\Gamma_{k-1}})$.  By Theorem~\ref{notallpathslift} there is a path $\gamma_P$ in $R_n$ that does not lift to $\Gamma_{k-1}$. Now write $\tilde s$ as a product $$s_0t_1s_1\cdot\ldots\cdot t_qs_q$$ where the $t_i$ are the path travelled in $\Gamma_{k-1}$ and where for $0<i< q$ $s_i$ is a nontrivial (reduced) edge-path which does not pass through any edges of $\Gamma_{k-1}$.  Then the paths $t_i$ do not contain a subpath that maps to $\gamma_P$  as we assume that $\gamma_P$ does not lift to $\Gamma_{k-1}$. The choice of $\Gamma_0$ (see condition (3) above) implies  that $$\sum\limits_{i=1}^q|t_i|\ge \frac{\beta}{6n}\cdot |s|.$$ 

Proposition~\ref{manysubpaths} implies that if $s$ lies in the generic set $S_{\gamma_P}$ then $q\ge \delta\cdot |s|$ where $\delta$ only depends on $\gamma_P$, $n$ and $\alpha$.
By removing a finite subset from $S_{\gamma_P}$ (which does not affect genericity of $S_{\gamma_P}$) we may assume that for every $s\in S_{\gamma_P}$ we have $\delta |s|-1\ge 1$. 

Then it follows that for the above decomposition of $\tilde s$ there exists $i\in \{1,\dots, q-1\}$ such that $|s_i|\le |s|/(\delta |s| -1)$. Indeed, if for all $i=1,\dots, q-1$ $|s_i|>|s|/(\delta |s| -1)>0$ then 
\[
|s|\ge \sum_{i=1}^{q-1} |s_i| > (q-1) |s|/(\delta |s| -1)\ge (\delta |s| -1)|s|/(\delta |s| -1)=|s|
\]
which is a contradiction.

Thus we can find $i_0\in \{1,\dots, q-1\}$ such that $|s_{i_0}|\le |s|/(\delta |s| -1)$.  Since $\lim_{N\to\infty} N/(\delta N -1)=1/\delta>0$, by removing a finite set from $S_{\gamma_P}$ we may assume that for every $s\in S_{\gamma_P}$ we have $|s|/(\delta |s| -1)\le 2/\delta$.  Thus $|s_{i_0}|\le |s|/(\delta |s| -1)\le 2/\delta$.

 By construction, the initial and terminal vertices of $s_{i_0}$ are in $\Gamma_{k-1}=\Theta$, but $s_{i_0}$ does not traverse any edges of $\Gamma_{k-1}$.  Therefore  for $\Gamma_k:=\Gamma_{k-1}\cup s_{i_0}$ we have $b(\Gamma_k)\ge b(\Gamma_{k-1})+1\ge (k-1)+1=k$. Thus the claim follows by taking $\ell_{P}:=\ell_{k-1}+\frac{2}{\delta}$ and $S_P=S_{k-1}\cap S_{\gamma_P}$. Thus Claim~B is verified, which completes the proof of Theorem~\ref{readingranndomwordsingraphs}.
\end{proof}

\section{Words representing $b_i$.}\label{S:generic}

Recall that we consider groups given by presentations of type 
\begin{equation*}
G=\langle a_1,\ldots ,a_n,b_1,\ldots ,b_n|a_i=u_i(\underline{b}),b_i=v_i(\underline{a})\rangle\tag{*}\end{equation*} where $n\ge 2$ and the $u_i$ and $v_i$ are freely reduced words of length $N$ for some (large) $N$.
If we now take the word/relator $a_i^{-1}u_i(\underline{b})$ and replace all occurrences of $b_i^{\pm 1}$ by $v_i(\underline{a})^{\pm 1}$ we obtain a relator in the $a_i$. We denote by $U_i$ the word obtained from this word by free and cyclic cancellation. A simple application of Tietze transformations shows that \begin{equation*}G= \langle a_1,\ldots ,a_n\mid U_1,\ldots ,U_n\rangle.\tag{**}\end{equation*}

We need an appropriate notion of genericity when working with groups given by presentation (*).  For $A=\{a_1,\dots, a_n\}$, with $n\ge 2$, we think of $F(A)$ as the set of all freely reduced words over $A^{\pm 1}$ and we use the same convention for $F(B)$ where $B=\{b_1,\dots, b_n\}$.  Thus for $N\ge 1$ there are exactly $2n(2n-1)^{N-1}$ elements of length $N$ in $F(A)$ (and same for  $F(B)$).  For $N\ge 1$ we denote by $T(A,N)$  the set of all $n$-tuples $(v_1,\dots, v_n)\in F(A)$ such that $|v_1|=\dots =|v_n|=N$. Similarly, $T(B,N)$ denotes the set of all $n$-tuples $(u_1,\dots, u_n)\in F(B)$ such that $|u_1|=\dots =|u_n|=N$. Thus for every $N\ge 1$ we have 
\[
\# T(A,N)=\#T(B,N)=\left(2n(2n-1)^{N-1}\right)^n=(2n)^n(2n-1)^{n(N-1)}.
\]

We denote by $T(A,B,N)$ the set of all $2n$-tuples $(v_1,\dots, v_n, u_1,\dots, u_n)$ such that $(v_1,\dots, v_n)\in T(A,N)$ and $(u_1,\dots, u_n)\in T(B,N)$. 
We also set $T(A)=\cup_{N=1}^\infty T(A,N)$, $T(B)=\cup_{N=1}^\infty T(B,N)$ and $T(A,B)=\cup_{N=1}^\infty T(A,B,N)$.  

\begin{defn}[Generic sets]

A subset $S\subseteq F(A)$ is \emph{generic} in $F(A)$ if
\[
\lim_{N\to\infty} \frac{\#\{v\in S : |v|=N \}  }{\#\{v\in F(A) : |v|=N \}  }=\lim_{N\to\infty} \frac{\#\{v\in S : |v|=N \}  }{2n(2n-1)^{N-1} }=1
\]
A property of elements of $F(A)$ is said to \emph{hold generically} if the set of all $v\in F(A)$ satisfying this property is a generic subset of $F(A)$.

A subset $Y\subseteq T(A)$ is called \emph{generic} in $T(A)$ if
\[
\lim_{N\to\infty} \frac{\#(Y\cap T(A,N)) }{\# T(A,N)}=\lim_{N\to\infty} \frac{\#(Y\cap T(A,N))     }{(2n)^n(2n-1)^{n(N-1)} }=1
\]

The notions of genericity for subsets of $F(B)$ and of $T(B)$ are defined similarly.

A subset $\mathcal T\subseteq T(A,B)$ is called \emph{generic} if
\[
\lim_{N\to \infty} \frac{\#(\mathcal T \cap T(A,B,N))}{\# T(A,B,N)}=\lim_{N\to\infty} \frac{\#(\mathcal T \cap T(A,B,N))}{(2n)^{2n}(2n-1)^{2n(N-1)}} = 1.
\]
\end{defn}

\begin{defn}[Genericity for presentations]$ $

We say that some property $\mathcal P$ for groups given by presentation (*) \emph{holds generically} if the set 
$\mathcal T_\mathcal P$ of all $(v_1,\dots, v_n,u_1,\dots, u_n)\in T(A,B)$, such that the group $G=\langle a_1,\ldots ,a_n,b_1,\ldots ,b_n|a_i=u_i,b_i=v_i, \text{ for } i=1,\dots, n\rangle$ satisfies $\mathcal P$, is a generic subset of $T(A,B)$. 
\end{defn}

The following fact is a straightforward consequence of known results
about genericity.  Lemma~\ref{wordsoccuratmostonce} below is
essentially a reformulation of the fact that a generic presentation
with a fixed number of generators and a fixed number of defining
relations satisfies arbitrarily strong small $C'(\lambda)$ cancellation condition, see~\cite{AO,A,KS,M}.

\begin{lem}\label{wordsoccuratmostonce} 
Let $0<\alpha<1$ be arbitrary. Then there is a generic subset
$\mathcal T\subseteq T(A,B)$ such that for every $\tau=(v_1,\dots, v_n, u_1,\dots, u_n)\in \mathcal T$
  with $|v_1|=\dots =|v_n|=|u_1|=\dots =|v_n|=N$ the following hold:

For every freely reduced $z\in F(A)$ with $|z|\ge \alpha N$
  there exists at most one $i\in \{1,\dots, n\}$ and at most one $\epsilon\in \{1,-1\}$ such that $z$ is a
  subword of $v_i^{\epsilon}$. Moreover, in this case there  is at most one occurrence of $z$ inside
  $v_i^{\epsilon}$. Also, the same property holds for every  freely reduced $z'\in
  F(B)$ with $|z'|\ge \alpha N$ with respect to the words $u_1,\dots,
  u_n$.

\end{lem}

Lemma~\ref{wordsoccuratmostonce} implies in particular that generically there is very little cancellation between the $v_i(\underline{a})^{\pm 1}$, and thus $U_i$ is essentially just the concatenation of the $v_i(\underline{a})^{\pm 1}$ introduced. The following statement is a straightforward consequence of standard small cancellation arguments; see~\cite{AO,A,KS,M} for similar arguments.

\begin{lem}\label{lemU_i} Generically the following hold:
\begin{enumerate}
\item $|U_i|\ge N^2(1-\frac{1}{10^{10}n})$ for $1\le i\le n$.
\item If $W$ is a subword of length $N^2\cdot \frac{1}{10^{10}n}$ of both $U_i^ {2\varepsilon_1}$ and  $U_j^ {2\varepsilon_2}$ whose initial letter lies in the first half of $U_i^ {2\varepsilon_1}$, respectively $U_j^ {2\varepsilon_2}$, then $i=j$, $\varepsilon_1=\varepsilon_2$ and the two occurrences of $W$ in $U_i^ {\varepsilon_1}$ coincide, i.e. start at the same letter of $U_i^ {\varepsilon_1}$. 
\end{enumerate}
\end{lem}

We refer the reader to \cite{LS,O,Stre} for basic results of small cancellation theory and basic properties of small cancellation groups.

Lemma~\ref{wordsoccuratmostonce}  implies that for any fixed $L\ge 1$, generically, the above presentation (*) is a $C'(\frac{1}{L})$-small cancellation presentation (note that the defining relations in (*) are already cyclically reduced).  This implies that for any $\alpha\in [0,1)$ we may assume that for  any word in the $a_i^{\pm 1}$ and $b_i^{\pm 1}$ that represents the trivial element in $G$ there exists a subword $R$ such that  $|R|\ge \alpha |RT^{-1}|$ where $RT^{-1}$ is a cyclic conjugate of a defining relator. In this case we call the process that replaces the subword $R$ by (its complementary word) $T$ an \emph{$\alpha$-small cancellation move} (relative to the relator $RT^{-1}$). We will also consider $\alpha$-small cancellation moves relative to relators that are not cyclic conjugates of defining relators of (*); in particular relative to the $U_i$, the defining relators of (**).

\begin{lem}\label{wordsrepresentingb_i} Let $\alpha\in [0,1)$. Let $w$ be a reduced word in the $a_i^{\pm 1}$ such that $w=_Gb_i$. Then for generic $G$ either $w=v_i$ or $w$ admits a $\alpha$-small cancellation move relative to some cyclic conjugate of some $U_i^{\pm 1}$.

In particular $w$ can be transformed into $v_i$ by applying finitely many $\alpha$-small cancellation moves relative to cyclic conjugates of the relators $U_i^{\pm 1}$ and by free cancellation.
\end{lem}

\begin{proof} This lemma can also be proven using the small cancellation properties of (**); however we argue using the presentation (*). 

The above discussion implies that the word $b_i^{-1}w$ can be transformed into the trivial word by replacing subwords that are almost whole relators by the complementary word of the relator. If $w\neq v_i$ then this process can initially only involve small cancellation moves relative to the relators $b_i^{-1}v_i(\underline{a})$ as $w$ is a word in the $a_i$ and therefore $b_i^{-1}w$ does not contain long subwords of the relators $a_i^{-1}u_i(\underline{b})$. 

These initial moves introduce  letters of type $b_j^{\pm 1}$ and possible remove the initial letter $b_i^{-1}$. Subsequent move of that type cannot  remove the $b_j^{\pm 1}$ that were introduced as otherwise the original word $w$ would have been non-reduced contradicting the hypothesis.

We perform these moves as long as possible, in the end we must have a long word in the $b_i$ with at most one $a_i$ in between that allows for a small cancellation move with respect to some relation $a_i^{-1}u_i(\underline{b})$ but then before the original moves almost all of a cyclic conjugate of $U_i$ must have occurred as a subword, indeed the origianl subword can be obtainded by resubstituing the $v_i(\underline{a})$ for the $b_i$, which gives almost all of some $U_i$ by definition of the $U_i$.
\end{proof}

We also need to establish several basic algebraic properties of groups
$G$ given by generic presentation (*). The proof of the following theorem follows the argument of Arzhantseva
and Ol'shanskii~\cite{AO,A} (see also \cite{KS}).

\begin{thm}\label{thm:free}
Let $n\ge 2$. Then for the group $G$ given by a generic presentation (*)
the following hold:
\begin{enumerate}
\item Every subgroup of $G$, generated by $\le n-1$ elements of $G$, is
  free.
\item $rank(G)=n$.

\item $G$ is torsion-free, word-hyperbolic and one-ended.

\end{enumerate}

\end{thm}

\begin{proof}

By genericity of (*), given any $0<\lambda<1$ we may assume that $(*)$
satisfies the $C'(\lambda)$ small cancellation condition.

(1) Choose a subgroup $H\le G$ of rank $\le n-1$.
Thus $H=\langle h_1,\dots, h_k\rangle\le
G$ for some $h_1,\dots, h_k\in G$ such that $k\le n-1$. 

Therefore there exists a finite connected graph $\Gamma$ (e.g. a
wedge of $k$ circles labelled by reduced words in $A\cup B$ representing
$h_1,\dots, h_k$)  with
positively oriented edges
labelled by elements of $A\cup B$ such that the first Betti number of
$\Gamma$ is $\le k$ and such that for the natural ``labeling
homomorphism'' $\mu: \pi_1(\Gamma)\to G$ the subgroup
$\mu(\pi_1(\Gamma))\le G$ is conjugate to $H$ in $G$. Among all 
such graphs choose the graph $\Gamma$ with the smallest number of
edges. Then, by minimality $\Gamma$ is a folded core graph.

Since $b(\Gamma)\le n-1$, Theorem~\ref{readingranndomwordsingraphs} is
applicable to $\Gamma$.

If $\mu:\pi_1(\Gamma)\to G$ is injective, then $H\cong\pi_1(\Gamma)$
is free, as required. Suppose therefore that $\mu$ is
non-injective. Then there exists a nontrivial closed circuit in
$\Gamma$ whose label is a cyclically reduced word equal to $1$ in
$G$. Therefore there exists a path $\hat\gamma$ in $\Gamma$ such that the
label $\hat z$ of $\hat\gamma$  is a subword of some a cyclic permutation of
some defining relation $r$ of (*) and that $|\hat z|\ge (1-3\lambda)|r|\ge
(1-3\lambda)N$. 

Without loss of generality, and up to a possible re-indexing, we may assume that $r$ is a cyclic
permutation of $b_1^{-1}v_1(\underline{a})$. Then $\hat z$ contains a
subword $z$ such that $z$ is also a subword of $v_1$ and that
$|z|\ge (\frac{1}{2}-\frac{3}{2}\lambda)|r| -1$. 
Thus we may assume that $|z|\ge \frac{9}{20}|r|$.

Let $\gamma$ be the
subpath of $\hat \gamma$ corresponding to the occurrence of $z$ in
$\hat z$. This implies that $|\gamma|=|z|\ge  \frac{9}{20}|r|$, so
that $|r|\le \frac{20}{9}|\gamma|$.

By genericity of (*) and by
Theorem~\ref{readingranndomwordsingraphs} we may assume, given an
arbitrary $0< \alpha_0<1$, that $\hat\gamma$ is
$\alpha_0$-injective. Since $\gamma$ comprises almost a half of $\hat
\gamma$, it follows that almost all edges of $\gamma$ are
$\hat\gamma$-injective.  Therefore, for any fixed $0<\alpha<1$, no
matter how close to $1$, we may
assume that the portion of $\gamma$ corresponding to
$\hat\gamma$-injective edges of $\gamma$ covers the length $\ge \alpha|\gamma|$.
We choose $0<\alpha<1$ and $0<\lambda<1$ so that $(9n-12)\lambda \frac{20}{9}<\alpha$.

There are at most $3(n-1)-3=3n-6$ maximal arcs of
$\Gamma$. These maximal arcs subdivide $\gamma$ as a concatenation of
non-degenerate arcs $\gamma=\gamma_0'\gamma_1'\dots \gamma_m'$ where for $i\ne 0,m$
$\gamma_i'$ is a maximal arc of $\Gamma$ and for $i=0,m$ $\gamma_i'$ is
either maximal or semi-maximal arc in $\Gamma$. The paths
$\gamma_0'$ and $\gamma_m'$ may, a priori, overlap in $\Gamma$ which
happens if the initial vertex of $\gamma$ is contained in the
interior of $\gamma_m'$ and the terminal vertex of $\gamma$ is
contained in the interior of $\gamma_0'$. By subdividing $\gamma_0'$ and
$\gamma_m'$ along the initial and terminal vertices of $\gamma$ if
needed, we obtain a decomposition of $\gamma$ as a concatenation of
nondegenerate arcs of $\Gamma$
\[
\gamma=\gamma_0\dots \gamma_k
\]
such that for $2\le i\le k-3$ the path $\gamma_i$ is a maximal arc of $\Gamma$
and such that for $i\ne j$ either $\gamma_i=\gamma_j^{\pm 1}$ or
$\gamma_i$ and $\gamma_j$ have no topological edges in common.
Let $\tau_1,\dots, \tau_q$ be all the distinct arcs in $\Gamma$ given
by the paths $\gamma_0\dots \gamma_k$
Since $\Gamma$ has $\le 3n-6$ maximal arcs, we have $q\le 3n-4$. 
Recall that the sum of the lengths of $\gamma_i$ such that the arc
$\gamma_i$ is $\hat\gamma$-injective, is $\ge \alpha
|\gamma|$. 

Suppose first that there exists $i_0$ such that $\gamma_{i_0}$ is
$\hat\gamma$-injective and that $|\gamma_{i_0}|> 3\lambda |r|$. Then we
remove the arc corresponding to $\gamma_{i_0}$ from $\Gamma$ and add
an arc $\tau$ from the terminal vertex of $\hat\gamma$ to the initial
vertex of $\hat\gamma$ with label $\hat y$ such that $\hat z\hat y$ is
a cyclic permutation of $r$.  Thus $|\hat y|=|\tau|<
3\lambda|r|$. Denote the new graph by $\Gamma'$. Then the image of the
labelling homomorphism $\pi_1(\Gamma')\to G$ is still conjugate to
$H$, but the number of edges in $\Gamma'$ is smaller than the number
of edges in $\Gamma$. This contradicts the minimal choice of
$\Gamma$. Therefore no $i_0$ as above exists.

Thus for every $i$ such that $\gamma_i$ is a $\hat\gamma$-injective
arc we have $|\gamma_i|\le 3\lambda |r|$. Hence the total length
covered by the $\hat\gamma$-injective $\gamma_i$ is 
\[
\le (3n-4)\cdot 3\lambda |r| \le (9n-12)\lambda \frac{20}{9}|\gamma| <
\alpha |\gamma|,
\]
yielding a contradiction.

Therefore $\mu:\pi_1(\Gamma)\to G$ is injective, and $H\cong
\pi_1(\Gamma)$ is free, as required.

(2) From presentation(*) we see that $G=\langle a_1,\dots, a_n\rangle$. Thus $rank(G)\le n$. Suppose that $k:=rank(G)<n$. Then, by part (1), $G$ is free
of rank $k<n$.

Hence there exists a graph $\Gamma$ with edges labelled by elements of
$A\cup B$ such that the first betti number of $\Gamma$ is $\le k$ and
such that the image of the labelling homomorphism
$\mu:\pi_1(\Gamma)\to G$ is equal to $G$. Among all such graphs choose
$\Gamma$ with the smallest number of edges. By minimality, $\Gamma$ is
a folded core graph.

Since the first betti
number of $\Gamma$ is $\le k<n$, there exists a letter $x\in A\cup B$
such that $\Gamma$ has no loop-edge labelled by $x$. Without loss of
generality we may assume that $x=a_1$. Since the image of $\mu$ is
equal to $G$, there exists a reduced path $\tilde\gamma$ in $\Gamma$ of length $>2$
whose label $\tilde z$ is a freely reduced word equal to $a_1$ in $G$.  Thus $\tilde za_1^{-1}=_G
1$. Whether or not the word $\tilde za_1^{-1}$ is freely reduced, it follows
that $\tilde\gamma$ contains a subpath $\gamma$ with label $z$ such
that $z$ is also a subpath of a cyclic permutation of some defining
relation $r$ of (*) with $|\gamma|=|z|\ge (1-3\lambda)|r|-1$. Then
exactly the same argument as in the proof of (1)
yields a contradiction with the minimal choice of $\Gamma$. 

Thus $rank(G)<n$ is impossible, so that $rank(G)\ge n$. Since we have
already seen that $rank(G)\le n$, it follows that $rank(G)=n$, as required.

(3) The group $G$ is given by presentation (**) as a proper quotient of $F_n$. Since $F_n$ is Hopfian, it follows that $G$ is not isomorphic to $F_n$.  The fact that $rank(G)=n$ implies that $G$ is not isomorphic to $F_r$ for any $1\le r\le n-1$. Thus $G$ is not free. 

Also, since (*) is a $C'(1/6)$ small cancellation presentation for $G$ where the defining relators are not proper powers, basic small cancellation theory~\cite{LS} implies that $G$ is a torsion-free non-elementary word-hyperbolic group. 

We claim that $G$ is one-ended. Indeed, suppose not. Since $G$ is torsion free and not cyclic, Stalling's theorem then implies that $G$ is freely decomposable, that is $G=G_1\ast G_2$ where both $G_1$ and $G_2$ are nontrivial. Therefore by Grushko's Theorem, $n=rank(G)=rank(G_1)+rank(G_2)$ and hence $1\le rank(G_1), rank(G_2) \le n-1$. Part (1) of the theorem then implies that $G_1$ and $G_2$ are free. Hence $G=G_1\ast G_2$ is also free, which contradicts part (2).
Thus $G$ is indeed one-ended, as claimed.\end{proof}

\begin{rem}
The proof of Theorem~\ref{thm:free} can be elaborated further,
following a similar argument to that of \cite{KS}, to prove the
following: For $G$ given by a generic presentation (*), 
if $g_1,\dots, g_n\in G$ are such that $\langle g_1,\dots,
g_n\rangle =G$ then $(g_1,\dots, g_n)\sim_{NE} (a_1,\dots, a_n)$ or
$(g_1,\dots, g_n)\sim_{NE} (b_1,\dots, b_n)$ in $G$.
\end{rem}

\section{Reduction moves}\label{reductionmoves}

Recall that for $A=\{a_1,\dots, a_n\}$ we think of $F(A)$ as the set of all freely reduced words over $A^{\pm 1}$.

We say that a word $z\in F(A)$ is a \emph{$U$-word} if $z$ is a subword of of a power of some $U_i^{\pm 1}$ (where $U_i$ are the defining relators for the presentation (**)).
Recall that by Lemma~\ref{lemU_i}  any $U$-word of length at least $N^2\cdot \frac{1}{10^{10}n}$ is a subword of a power of only one $U_i^ {\pm 1}$. We say that a word $V$ is \emph{$U$-complementary} to a $U$-word $W$ if $WV^{-1}$ is a power of a cyclic conjugate of some $U_i^{\pm 1}$. Note that $V$ is then also a $U$-word. Note also that for a given $U$-word $W$, its $U$-complementary word is not uniquely defined. Thus if $W$ is an initial segment of $U_*^k$, where $k\ne 0$ and $U_*$ is a cyclic permutation of $U_i^{\pm 1}$, and $V$ is $U$-complimentary to $W$ then for every integer $m$ such that $mk<0$ the word $U_*^mV$ is also $U$-complimentary to $W$.

\smallskip We now introduce a partial order on $F(A)$.  First, for $w\in F(A)$ we define $c_1(w)$ to be the minimum $k$ such that $w$ can be written as a reduced product of type $w_1\cdot \ldots \cdot w_k$ where each $w_i$ is a $U$-word. We call any such decomposition of $w$ as a product of $c_1(w)$ $U$-words an \emph{admissible decomposition} and we call the words $w_1,\dots, w_k$ the \emph{$U$-factors} of this admissible decomposition.

\smallskip The number $c_1(w)$ can be interpreted geometrically in the Cayley complex $CC$ of the presentation (**): If $k=c_1(w)$, $w_1\cdot \ldots \cdot w_k$ is as above and  $\gamma_w$ is a path reading $w$ then each $w_i$ is read by a subpath of $\gamma_w$ that travels (possibly with multiplicity) along the boundary of a 2-cell of $CC$.

\smallskip Note that admissible decompositions are not unique. However the following statement shows that the degree of non-uniqueness for admissible decompositions can be controlled:

\begin{lem}\label{uniquenessofwi}
Let $w\in F(A)$, $k=c_1(w)$ and let $w=w_1\cdot\ldots\cdot w_k=w_1'\cdot\ldots\cdot w_k'$ be admissible decompositions of $w$.  Then the following hold:

\begin{enumerate}
\item The subwords $w_i$ and $w_i'$ overlap non-trivially in $w$ for all $i\in\{1,\ldots ,k\}$.
\item If $V$ is the minimal subword of $w$ that contains both $w_i$ and $w_i'$ then both $w_i$ and $w_i'$ contain the subword $V'$ obtained from $V$ by removing the initial and terminal subword of length $\frac{1}{10^{10}n}N^2$. 
\item If $w_i$ is subword of a power of $U_j^{\varepsilon}$ of length at least $\frac{1}{10^5n}N^2$ then $w_i'$ is also a subword of a power of $U_j^{\varepsilon}$ and the common subword of $w_i$ and $w_i'$ differs from $w_i$ and $w_i'$ by at most an initial and terminal segment of length $\frac{1}{10^{10}n}N^2$.
\end{enumerate}
\end{lem}

\begin{proof} (1) Suppose that this does not hold, i.e. that for some $i\in\{1,\ldots ,k\}$ the words $w_i$ and $w_i'$ do not overlap. Without loss of generality we assume that $w_i$ comes before $w_i'$ when reading $w$. Therefore $w_i'w_{i+1}'\cdot\ldots w_k'$ is a subword of $w_{i+1}w_{i+2}\cdot\ldots\cdot w_k$,  and, in particular, $w_i' w_{i+1}'\cdot\ldots w_k'$ is a product of $k-i$ $U$-words. It then follows that $w$ is product of $(i-1)+(k-i)=k-1$ $U$-words, contradicting the definition of $c_1(w)$.

(2) Suppose that the conclusion of part (2) of the lemma does not hold. Without loss of generality we may assume that $w_1\cdot\ldots\cdot w_{i-1}$ has a suffix $\bar w_i$ of length at least $ \frac{1}{10^{10}n}N^2$ that is a prefix of $w_i'$. The subword $\bar w_i$ is in fact a suffix of $w_{i-1}$ as otherwise $w_{i-1}$ and $w'_{i-1}$ do not overlap, contradicting (1).

\smallskip Thus $w_{i-1}=u\bar w_i$ for some word $u$ and $w'_i=\bar w_iv$ for some word $v$. Now as $w_{i-1}$ and $w'_i$ are $U$-words and $\bar w_i$ is of length at least $ \frac{1}{10^{10}n}N^2$, it follows from Lemma~\ref{lemU_i}(2)  that $z=u\bar w_iv$ is a $U$-word. It follows that $$w=w_1\cdot\ldots w_{i-2}zw_{i+1}'\cdot\ldots w'_k,$$ contradicting the assumption that $k=c_1(w)$.

 (3) This part follows immediately from the arguments used to prove (1) and (2) together with Lemma ~\ref{lemU_i}(2).
\end{proof}

 Now let $w$ be a freely reduced word over $\{a_1,\dots, a_n\}^{\pm 1}$ with $c_1(w)=k$. For $1\le i\le k$ we will define a number $\ell_w(i)$ that essentially measures the length of the $i$-th factor in an admissible decomposition of $w$, provided this factor is almost of length $N^2$. The precise definition of $\ell_w(i)$ is independent of the chosen admissible decomposition of $w$ and is more robust under certain modifications of $w$ that we will introduce later on.

\medskip 

\begin{lem}\label{L:i-equiv} 
Let $w\in F(A)$ and $k=c_1(w)$. Fix  $i\in\{1,\ldots ,k\}$. Choose an admissible decomposition $w=w_1\cdot\ldots\cdot w_k$ of $w$. Suppose that for some $j\neq i$ the subword $w_j$  is a maximal $U$-subword of $w$  with $|w_j|\ge \frac{1}{10^3n}N^2$. Choose a  word $\bar w_j$ that is $U$-complementary to $w_j$  that is also of length at least $\frac{1}{10^3n}N^2$ and let $\bar w:=w_1\cdot\ldots w_{j-1}\bar w_jw_{j+1}\cdot\ldots \cdot w_k$.

Then the following hold:
\begin{enumerate}
\item $\bar w$ is freely reduced.
\item $\bar w_j$ is a maximal $U$-subword of $\bar w$.
\item $c_1(\bar w)=c_1(w)$.
\end{enumerate}
\end{lem}

\begin{proof}(1) and (2) follow from the maximality of the $U$-subword $w_j$ and the fact that $w$ is  reduced and that $w_j\bar w_j^{-1}$ is cyclically reduced. (3) follows  from (1) and (2) as the assumption on the length of $w_j$ and $\bar w_j$ together with Lemma~\ref{uniquenessofwi}(3)  imply that we can find admissible decompositions of $w$ and $\bar w$ such that $w_j$, respectively $\bar w_j$, occur as factors.
\end{proof}

For every $k\ge 1$ denote by $F(A;k)$ the set of all $w\in F(A)$ with $c_1(w)=k$.

\begin{defn}[$i$-equivalence]
If $w\in F(A;k)$, $i\in \{1,\dots, k\}$  and $\bar w$ is obtained from $w$ as in Lemma~\ref{L:i-equiv}, we say that $\bar w$ is \emph{elementary $i$-equivalent} to $w$.  Lemma~\ref{L:i-equiv} implies that $\bar w\in F(A;k)$.

We further say that $w\in F(A;k)$ is \emph{$i$-equivalent} to a word $w'\in F(A)$ if $w$ can be transformed into $w'$ by applying finitely many elementary $i$-equivalences (note that in this definition the value of the index $i$ is fixed). Thus for every  $i\in\{1,\ldots ,k\}$ $i$-equivalence is equivalence relation on $F(A;k)$. 
\end{defn}

\begin{lem}\label{i-equivalenceisharmless} Suppose that $w, w'\in F(A;k)$ are  $i$-equvalent for some $1\le i \le c_1(w)=c_1(w')$ and that  $w=w_1\cdot\ldots \cdot w_k$ and $w'=w'_1\cdot\ldots \cdot w'_k$ are admissible decompositions of $w$, respectively $w'$. Then the following hold:
\begin{enumerate}
\item There exists a word $V\in F(A)$ such that $w_i=xVy$ and $w_i'=x'Vy'$ where $x,y,x',y'$ are words of length at most $\frac{1}{10^{10}}\cdot N^2$.
\item The lengths of $w_i$ and $w_i'$ differ at most by $\frac{2}{10^{10}}\cdot N^2$.
\end{enumerate}
\end{lem} 
\begin{proof} This statement follows from the fact that an elementary $i$-equivalence that modifies a $w_j$ only affects the adjacent factors and increases or decreases their length by at most $\frac{1}{10^{10}n}N^2$ by Lemma~\ref{uniquenessofwi},  and that another such elementary $i$-equivalence applied to the $j$-th factor reverses theses changes.
\end{proof}

\begin{defn}[Complexity of a word in $F(A)$]

Let $w\in F(A)$, $k=c_1(w)$.
For each $i\in \{1,\dots, k\}$ we define $\hat\ell_w(i)$ to be the maximum over all $|w_i|$ where $w_i$ is the $i$-th factor in some admissible decomposition of a word $w'$ that is $i$-equivalent to $w$.
Note that part~(2) of Lemma~\ref{i-equivalenceisharmless} implies that $\hat\ell_w(i)<\infty$.

 We further define $\ell_w(i)=0$ if $\hat\ell_w(i)<(1-\frac{1}{10^3n})N^2$ and $\ell_w(i)=\hat\ell_w(i)$ otherwise. Note that $\ell_w(i)=\ell_{\bar w}(i)$ if $w$ and $\bar w$ are $i$-equivalent. We now define $$c_2(w):=\ell_w(1)+\ldots+\ell_w(c_1(w)).$$

\medskip For $w\in F(A)$ we define the \emph{complexity} of $w$ as $c(w)=(c_1(w),c_2(w))$ and order complexities lexicographically. This order is a well-ordering on $F(A)$
\end{defn}

\medskip It is not hard to verify that an $\alpha$-small cancellation move relative to $U_j$ for $\alpha$ close to $1$ decreases the complexity. We will need the following generalization of this fact:

\begin{lem}\label{reductionmove} Let $w\in F(A)$ be a reduced word and let $W\in F(A)$ be a $U$-word of length $\frac{1}{100n}N^2$. Suppose that $w$ has (as a word) a reduced product decomposition of type
\[
w=x_0W^{\varepsilon_1}x_1W^{\varepsilon_2}\cdot\ldots W^{\varepsilon_q}x_q
\] with $\varepsilon_t\in\{-1,1\}$ for $1\le t\le q$. Choose $V$ such that $WV^{-1}$ is a cyclic conjugate of some $U_j^{\pm 1}$ and let $w'$ be the word obtained from the word $x_0V^{\varepsilon_1}x_1V^{\varepsilon_2}\cdot\ldots V^{\varepsilon_q}x_q$ by free reduction. Then the following hold:
 \begin{enumerate}
\item  $c(w')\le c(w)$.
\item If one of the subwords $W^{\varepsilon_i}$ is contained in a $U$-subword  of $w$ of length at least $(1-\frac{1}{10^4n})N^2$, then $c(w')< c(w)$. 
\end{enumerate}
\end{lem}

\begin{proof} We will first observe that we may choose an admissible decomposition $w=w_1\cdot \ldots\cdot w_k$  (with $k=c_1(w)$)  such that any subword $W^{\varepsilon_t}$ of $w$ occurs inside some factor $w_j$ of $w$. Indeed in any admissible  decomposition of $w$ the subword $W^{\varepsilon_t}$ can overlap with at most $3$ $U$-factors of $w$, as otherwise an admissible decomposition of $w$ with fewer $U$-factors can be found. Thus $W^{\varepsilon_t}$ overlaps with some factor in a subword of length at least $\frac{1}{300n}N^2$. Because of  Lemma~\ref{uniquenessofwi},  we can expand this factor to contain $W^{\varepsilon_t}$. The claim follows by making this modifications for all $W^{\varepsilon_t}$.

Let now $w=w_1\cdot\ldots\cdot w_k$ be an admissible decomposition  such that any subword $W^{\varepsilon_i}$ occurs inside some factor $w_j$ of $w$.  
Note that possibly several $W^{\varepsilon_t}$ occur inside the same $w_j$. 
Note also that if some $W^{\varepsilon_t}$ is contained in a $U$-subword $x$ of length at least $(1-\frac{1}{10^4n})N^2$ then it is a subword of some $w_j$ of length at least $(1-\frac{2}{10^4n})N^2$. This conclusion can be established by first observing that by the argument before another admissible decomposition $w=z_1\cdot\ldots \cdot z_k$ (with $U$-factors $z_i$) can be chosen such that $x$ is contained in some $z_j$, and then applying Lemma~\ref{uniquenessofwi}. 

\smallskip  Replacing the $W^{\varepsilon_t}$ by the $V^{\varepsilon_t}$ inside each $w_j$ and performing free reduction yields a new word $w_j'$, which is again a $U$-word. The word $w'$ is then obtained from $w_1'\cdot\ldots\cdot w_k'$ by free reduction. This construction implies, in particular, that $c_1(w')\le c_1(w)$.

\smallskip Thus we may assume that $c_1(w)=c_1(w')=k$ as there is nothing to show otherwise. To prove the lemma it suffices to show that  $\ell_w(i)\ge \ell_{w'}(i)$ for all $1\le i\le k$ and that $\ell_w(i)> \ell_{w'}(i)$ for some $1\le i\le k$ if condition (2) is satisfied. 

If condition (2) of the lemma is satisfied, there exists some factor $w_i$ of length at least $(1-\frac{2}{10^4n})N^2$ that contains one of the  $W^{\varepsilon_t}$.
Therefore it is sufficient to consider the following two cases:

\smallskip\noindent {\bf Case 1: } Suppose that $w_i$ contains at least one of the words $W^{\varepsilon_t}$. Note that $\hat\ell_w(i)=\ell_w(i)\ge N^2$ if $w_i$ contains more than one such word. 

Note first that $\ell_{w'}(i)=0$ if $\ell_w(i)=0$, since the word obtained from $w_i$ be replacing $W^{\pm 1}$ with $V^{\pm  1}$ followed by free cancellation is of length at most $(1-\frac{1}{100n})N^2$ and since $i$-equivalence cannot increase this length by more than $\frac{2}{10^{10}n}N^2$ by Lemma~\ref{i-equivalenceisharmless}.

 If $\ell_w(i)\ne 0$ then the replacements of  $W^{\pm 1}$ with $V^{\pm  1}$ decrease the length of $w_i$ by more than $\frac{2}{10^{10}n}N^2$. Thus it follows from Lemma~\ref{i-equivalenceisharmless}  that $\ell_{w'}(i)<\ell_w(i)$. This holds in particular if $w_i$ is  of length at least $(1-\frac{2}{10^4n})N^2$,  which (by the above discussion) is guaranteed to occur if  condition (2) is satisfied.

\smallskip\noindent {\bf Case 2: } Suppose that $w_i$ contains no subword $W^{\varepsilon_t}$. To conclude the proof, we need to show that  $\ell_w(i)\ge \ell_{w'}(i)$. Note that for any admissible decomposition $z_1\cdot\ldots \cdot z_k$ the words $w_i$ and $z_i$ have a non-trivial common subword that is not touched by the replacement and free cancellations when going from $w$ to $w'$, since otherwise $c_1(w')<k$; this follows as in the proof of Lemma~\ref{uniquenessofwi}(1).

\smallskip If $w$ is $i$-equivalent to $w'$ then there is nothing to show as $\ell_w(i)= \ell_{w'}(i)$ by the definition of $\ell(i)$. However it may happen that $w$ and $w'$ are not $i$-equivalent as some of the $U$-words $w_j$ may be replaced by $U$-words that are too short for the replacements to qualify as $i$-equivalences. This is the reason why the argument for treating Case~2 is slightly more subtle.

\smallskip Choose a word $w''$ that is $i$-equivalent to $w'$ such that $w''$ realizes $\hat\ell_{w'}(i)$, i.e. that $w''$ has an admissible decomposition $w''_1w''_2\cdot\ldots \cdot w''_k$ such that $|w''_i|=\hat\ell_{w'}(i)$. It  suffices to establish:

{\bf Claim:} The word $w$ is $i$-equivalent to a word $\tilde w$ that has an admissible  decomposition $\tilde w_1\tilde w_2\cdot\ldots \cdot \tilde w_k$  such that $|\tilde w_i|=|w''_i|$. 

Indeed,  the existence of such a word $\tilde w$ implies that $$\hat\ell_{w}(i)\ge |\tilde w_i|=|w''_i|=\hat\ell_{w'}(i)$$ which implies that $\ell_w(i)\ge \ell_{w'}(i)$.

\smallskip To verify the Claim, we first construct a word $\dot w$ that is $i$-equivalent to $w$ in the following way: Start with $w$ and perform the same replacements as when going from $w$ to $w'$, but whenever one of the replacements results in a maximal $U$-subword $v$ of length less than $\frac{1}{3}N^2$, multiply it with a cyclic conjugate of the appropriate $U_j^{\pm 1}$ to obtain a $U$-word longer than $\frac{1}{3}N^2$ that has $v$ as both initial and terminal word and such that the replacement is an elementary $i$-equivalence.

 Call the resulting word $\dot w$. Note that $\dot w$ is $i$-equivalent to $w$ and all words that occur as the $i$-th factors in any admissible decomposition of $w'$ also occur as $i$-th factors in an admissible decomposition of $\dot w$. Now all elementary $i$-equivalences that need to be applied to $w'$ to obtain $w''$ can be applied to $\dot w$. We obtain a word that is $i$-equivalent to $\dot w$ and that has $w_i''$ as the $i$-th $U$-word. This proves the Claim and completes the proof of the lemma.\end{proof}

\section{Proof of the Main Theorem}\label{proofmaintheorem}

In this section we establish Theorem~\ref{thm:A} from the Introduction:

\begin{thm}\label{thm:main}
 Let  $n\ge 2$ be arbitrary.
Then for the group $G$ given by a generic presentation (*),  the group $G$ is torsion-free, word-hyperbolic, one-ended of $rank(G)=n$, and 
the $(2n-1)$-tuples $(a_1,\ldots ,a_n,1,\ldots 1)$ and $(b_1,\ldots ,b_n,1,\ldots 1)$ are not Nielsen-equivalent in $G$.
\end{thm}

Let $G$ be given by a generic presentation (*).   By Theorem~\ref{thm:free} we know that $G$ is torsion-free, word-hyperbolic, one-ended and that $rank(G)=n$.

For a $(2n-1)$-tuple $\tau=(w_1,\dots, w_{2n-1})$ of freely reduced words in $F(A)$ define the \emph{complexity} $c(\tau)$ of $\tau$
as \[c(\tau):=(c(w_1),\ldots ,c(w_n)).\]
 We order complexities lexicographically. Note that in this definition we completely disregard the words $w_{n+1}, \dots, w_{2n-1}$.

We argue by contradiction. Suppose that the conclusion of Theorem~\ref{thm:main} fails for $G$, so that the  $(2n-1)$-tuples $(a_1,\ldots ,a_n,1,\ldots 1)$ and $(b_1,\ldots ,b_n,1,\ldots 1)$ are Nielsen-equivalent in $G$.

 Then there exist elements $x_1, \ldots ,x_{2n-1}$ in $F(A)$ such that
 the $(2n-1)$-tuple $(x_1, \ldots ,x_{2n-1})$ is Nielsen equivalent in
 $F(A)$  to the $(2n-1)$-tuple $(a_1,\ldots ,a_n,1,\ldots ,1)$ and
 such that for some permutation $\sigma\in S_n$ we have $x_i=_G b_{\sigma(i)}$ for $1\le i\le n$.

 {\bf Minimality assumption:}
 Among all $(2n-1)$-tuples $(x_1, \ldots ,x_{2n-1})$ of freely reduced words in $F(A)$ with the above property choose a tuple 
 $(x_1, \ldots ,x_{2n-1})$ of minimal complexity.

This minimality assumption implies that $c(x_i)\le c(x_{i+1})$ for $i=1,\dots, n-1$. After a permutation of the $b_i$ we may moreover assume that $\sigma=id$.
We fix this minimal tuple $(x_1, \ldots ,x_{2n-1})$ for the remainder of the proof. The minimality assumption plays a crucial role in the argument and ultimately it will allow us to derive a contradiction.

\medskip We choose $k$ maximal such that $x_i=v_i$ in $F(A)$ for $1\le i\le k$. By convention if
$x_1\ne v_1$, we set $k=0$.

Note that if $k<n$ (which we will show below) then it follows from Lemma~\ref{wordsrepresentingb_i} that $x_{k+1}$ admits an $ \alpha$-small cancellation move relative to some $U_j$ for $\alpha$ close to~$1$. This implies in particular that we may assume that $x_{k+1}$ is of length at least $\frac{1}{99n}N^2$.

\medskip Let now $\Gamma$ be the wedge of $2n-1$ circles, wedged at a
vertex $p_\Gamma$, such that the $i$-the circle is labeled by the word $x_i$ for $1\le i\le 2n-1$. By assumption the $x_i$ form a generating set of $F(A)$, thus the natural morphism $f$ from $\Gamma$ to the rose $R_n$ is $\pi_1$-surjective as the loops of $\Gamma$ map to a generating set of $\pi_1(R_n)$. It follows that $\Gamma$ folds onto $R_n$ by a finite sequence of Stallings folds. Thus there exists a sequence of graphs $$\Gamma=\Gamma_0,\Gamma_1,\ldots ,\Gamma_l=R_n$$ and Stallings folds $f_i:\Gamma_ {i-1}\to \Gamma_i$ for $1\le i\le l$ such that $f=f_l\circ\ldots\circ f_1$. We may choose this folding sequence such that we do not apply a fold that produces a lift of $R_n$ in $\Gamma_i$ if another fold is possible. Choosing $\Delta$ to be the last graph in such a folding sequence that does not contain a lift of $R_n$ it follows that  there exists a graph $\Delta$ with the following properties:
\begin{enumerate}
\item There exists a map $f_{\Delta}:\Gamma\to\Delta$ such that $f:\Gamma\to R_n$  factors through $f_{\Delta}$.
\item $\Delta$ does not contain a lift of $R_n$.
\item Any fold applicable to $\Delta$ produces a lift of $R_n$, thus $\Delta$ contains a connected subgraph $\Psi$ with at most $n+2$ edges that folds onto $R_n$.
\end{enumerate}

Put $p_\Delta:=f_{\Delta}(p_\Gamma)$. Thus the petal-circles of $\Gamma$ get mapped to closed paths based at $p_\Delta$.

As the Betti number of $\Delta$ is bounded by $2n-1$ and as $\Delta$ folds onto $R_n$ it follows that $\Delta$ does not contain a 2-sheeted cover of $R_n$. Indeed such a 2-sheeted cover would be of Betti number $2n-1$ and therefore be the core of $\Delta$ which implies that $\Delta$ does not fold onto $R_n$. It follows that the hypothesis of Theorem~\ref{readingranndomwordsingraphs} is satisfied for $\Delta$. Thus for $\alpha$ arbitrarily close to $1$ we may assume that any lift of of the path $\gamma_{v_i}$ in $R_n$ with label $v_i=v_i(\underline a)$ to $\Delta$ is $\alpha$-injective.

\begin{lem} \label{lem:k<n}
We have $k<n$.
\end{lem}

\begin{proof} For $i\in\{1,\ldots ,k\}$ the $i$-th loop of $\Gamma$ maps to the path $\gamma_{v_i}$ in $R_n$. Thus the $i$-th loop of $\Gamma$ maps to a path in $\Delta$ that is a lift of $\gamma_{v_i}$, thus by  Theorem~\ref{readingranndomwordsingraphs} we can assume that this image in $\Delta$ is $\alpha$-injective for $\alpha$ arbitrarily close to $1$.

This implies that for $1\le i\le k$ the image of the $i$-th loop of $\Gamma$ in $\Delta$ contains an arc $\sigma_i$ of $\Delta$ of length at least $\frac{1}{100n}|v_i|=\frac{1}{100n}N$ as the image of $\sigma_i$ is the union of at most $6n$ arcs of $\Delta$. By genericity these arcs are travelled only once by the first $k$ loops, in particular they are pairwise distinct. Moreover they are disjoint from $\Psi$. It follows that there exists a connected subgraph $\bar\Psi$ of $\Delta$ containing $\Psi$ such that the arcs $\sigma_i$ meet $\bar\Psi$ only in their endpoints. As the Betti number of $\Psi$ is at least $n$ it follows that the the subgraph of $\Delta$ spanned by $\bar\Psi$ and the $\sigma_i$ has Betti number $n+k$. As $\Delta$ has Betti number at most $2n-1$ this implies that $k\le n-1<n$.
\end{proof}

We now continue with the proof of Theorem~\ref{thm:main}.

Recall that $U_i$ was obtained from the generic word $a_i^{-1}u_i(\underline{b})$ by replacing each  $b_j$ by the respective $v_j(\underline{a})$. The genericity of the $v_j$ implies that there is very little cancellation in a freely reduced word in the $v_j$, thus $U_i$ is essentially the concatenations of $a_i^{-1}$ and the words $v_j$ that  were introduced to replace the letters~$b_1,\dots,b_n$.

\smallskip For each $v_j$ let $v_j'$ be the subword of $v_j$ obtained by chopping-off small initial and terminal segments of $v_i$ and such that for every occurrence of $b_j$ in $a_i^{-1}u_i(\underline{b})$ the word $v_j'$ survives without cancellation when we replace $b_1,..,b_n$ by $v_1,..,v_n$ in $a_i^{-1}u_i(\underline{b})$ and then freely and cyclically reduce to obtain  $U_i$.

As the cancellation is small relative to the length of the $v_j$ it follows that $v_j'$ is  almost all of $v_j$.
Moreover, we can choose $v_j'$ so that $|v_1'|=\dots=|v_n'|=N'$. Choosing the  $v_j$ from the appropriate generic set, we may assume that $N\ge N'\ge N(1-\epsilon_0)$ for any fixed $0<\epsilon_0<1$.

By Theorem~\ref{readingranndomwordsingraphs} and by the genericity of the presentation (*), for any $0<\alpha_0<1$ we may assume that the conclusion of Theorem~\ref{readingranndomwordsingraphs} holds for the words $v_j$. It follows that for any given $0<\alpha<1$ we may also assume that the words $v_j'$  satisfy the conclusion of Theorem~\ref{readingranndomwordsingraphs} as well. Indeed if we choose $\alpha_0=\frac{1}{2}(1+\alpha)$ and $\epsilon_0=\frac{1}{2}(1-\alpha)$ then any image of $v_j'$ in some graph as in the hypothesis of Theorem~\ref{readingranndomwordsingraphs} contains at least $\alpha_0\cdot N-\epsilon_0\cdot N=\alpha\cdot N\ge \alpha\cdot |v_i'|=\alpha N'$ distinct topological edges, so that the path corresponding to $v_j'$ in the graph in question is $\alpha$-injective.

\smallskip It follows from Lemma~\ref{wordsrepresentingb_i} that $x_{k+1}$, the label of the $(k+1)$-st loop, admits an $\alpha$-small cancellation move with respect to some $U_i$ with $\alpha$ close to $1$. Thus $x_{k+1}$ contains a subword $Z$ of length at least $\left(1-\frac{1}{10^4n}\right)N^2$ that is  a subword of a cyclic conjugate of some $U_i$.  After dropping a short suffix and prefix we can assume that $Z$ is a reduced product of remnants of some $v_j^{\pm 1}$ in $U_i$, and each such remnant contains the corresponding $(v_j')^{\pm 1}$. We refer to the sub-words $(v_j')^{\pm 1}$ of $Z$ (understood both as words and as positions in $Z$ where they occur) arising in this way as \emph{$v$-syllables} of $Z$. 
Denote the arc in the $(k+1)$-st loop of $\Gamma$ that corresponds to $Z$ by $\tilde s_Z$.  Note that $Z$ is readable in $\Delta$ since $Z$ is the label of the path $s_Z:=f_\Delta(\tilde s_Z)$. 

\begin{conv}\label{conv:Z}
We fix the word $Z$, the arc $\tilde s_Z$ in $\Gamma$ and its image $s_Z$ in $\Delta$ for the remainder of the proof of Theorem~\ref{thm:main}.
\end{conv}

In the following two lemmas we will establish some important
properties of the word $Z$ and of the path $s_Z$.

\begin{lem}\label{L:Claim}
For $1\le j\le n$ there exists a non-trivial distinguished subword $\tilde v_j$ of $v_j'$ (this means both the abstract word and the position in $v_j'$ where it occurs) such that the following hold:
\begin{enumerate}
\item  For any $v$-syllable $(v_j')^\epsilon$ of $Z$ with corresponding subword $\tilde v_j^\epsilon$ 
the subpath of $s_Z$ corresponding to $\tilde v_j^\epsilon$ is contained in an arc of $\Delta\backslash \Psi$. (This means the word $\tilde v_j^\epsilon$ is read on an arc of $\Delta\backslash \Psi$).
\item If $1\le j_1,j_2\le n$, $j_1\ne j_2$ and $(v_{j_1}')^\epsilon$ and  $(v_{j_2}')^\delta$ are two $v$-syllables of $Z$, then the subpaths of $s_Z$ corresponding to their distinguished subwords  $\tilde v_{j_1}^\epsilon$ and
$\tilde v_{j_2}^\delta$ have no topological edges in common.

\item Similarly, if $1\le j\le n$ and $v'$, $v''$ are two distinct (i.e. appearing in different positions in $Z$) $v$-syllables of $Z$ both representing $(v_j')^{\epsilon}$ and $(v_j')^{\delta}$, then the subpath of $s_Z$ corresponding to the distinguished subword $(\tilde v_j)^{\epsilon}$, $(\tilde v_j)^{\delta}$ 
are either disjoint or coincide if $\epsilon=\delta$ and are inverses if $\epsilon=-\delta$.
\end{enumerate}

\end{lem}

\begin{proof}

For a finite graph $\Lambda$ let $f(\Lambda)$ be the number of paths $e_1,\ldots ,e_k$ in $\Lambda$ such that $\{e_i,e_i^ {-1}\}\neq \{e_j,e_j^ {-1}\}$ for $i\neq j\in \{1,\ldots ,k\}$ and $\{i,j\}\neq \{1,k\}$. Thus $f(\Lambda)$ is the number of path that travel no topological edge twice except possibly that the first and the last edges may agree. Let $f(n)$ be the maximal $f(\Lambda)$ where $\Lambda$ is a graph with the following properties:
\begin{enumerate}
\item There exists a vertex $v\in V\Lambda$ such that $(\Lambda,v)$
  is a core graph with respect to $v$  of Betti number at most $2n-1$.
\item  $\Lambda$ has no valence-2 vertex.  
\end{enumerate} 

Choose $\alpha=\alpha(n)<1$ (close to 1) and $d=d(n)>0$ (close to 0) such that $$f(n)\cdot  \left((1-\alpha)+6\cdot n\cdot d(n)\right)\le \frac{1}{2}.$$  
As noted above, by genericity we may assume that any path corresponding to $v_j'$ is $\alpha$-injective. Moreover by Lemma~\ref{wordsoccuratmostonce} we may assume that $v_j'$ does not contain any word of length $d(n)\cdot |v_j'|$ twice. Let $\Psi'$ be the union of all maximal arcs of $\Delta$ of length less than  $d(n)|v_j'|=d(n)N'$. Note that $\Psi'$ contains $\Psi$.

For $j=1,\dots, n$ we say that a reduced decomposition
\[
v_j'=y_0p_1y_1\dots p_m y_m \tag{!}
\]
is \emph{basic} if $m\ge 1$ and there exists a way to read $v_j'$ in $\Delta$ such that the following hold:
\begin{enumerate}
\item $|p_t|\ge d(n)\cdot |v_j'|=d(n)N'$ for $1\le i\le m$ and $p_t$ is read on an arc of $\Delta\setminus \Psi'$.
\item For $1\le t\le m-1$ each $y_t$ is read $\Psi'$.

\item $y_0=y_0'y_0''$ where $y_0''$ is read in $\Psi'$, $y_0'$ is read on an arc (possibly degenerate) of $\Delta\setminus \Psi'$ and $|y_0'|<d(n)N'$.
\item $y_m=y_m'y_m''$ where $y_m'$ is read in $\Psi'$, $y_m''$ is read on an arc of $\Delta\setminus \Psi'$ and $|y_m''|<d(n)N'$.
\end{enumerate} 

For a basic decomposition (!) of $v_j'$, we say that a path $\beta$ in $\Delta$ corresponding to reading $v_j'$ in $\Delta$ as in
the above definition, is a \emph{basic path}  representing~(!).

Note that the $y_t$ may or may not be degenerate, i.e. may be single vertices of $\Delta$. Note that whenever $v_j'$ can be read in $\Delta$ then the corresponding path is a basic path corresponding to a basic decomposition. This implies that any subpath of $s_Z$ corresponding to a subword of type $v_j'^{\pm 1}$ is a basic path corresponding to a basic decomposition.

Note that (2) implies that that the word $p_t$ is read along a maximal arc of $\Delta\setminus \Psi'$ if $2\le t\le m-1$.

Moreover for any $t_1\ne t_2$ and $\{t_1,t_2\}\neq \{1,m\}$ the subpaths of $\beta$ corresponding  to $p_{t_1}$ and $p_{t_2}$ have no topological edges in common by the choice of $d(n)$. The subpath corresponding to  $p_m$ can overlap the subpath corresponding to $p_1$ in a subpath of length at most $d(n)N'$ if they lie on the same arc of $\Delta\setminus \Psi'$.

Let now $\hat\Delta$ be  the graph obtained from $\Delta$ by collapsing all connected components of $\Psi'$ to points. If (!) is a basic decomposition of $V_j'$ and $\beta$ is a basic path in $\Delta$ representing (!), then the image of $\beta$ in $\hat \Delta$ is a path $\hat \beta$ with label $y_0'\cdot p_1\cdot \ldots \cdot p_m\cdot y_m''$.  Moreover, the path $\hat\beta$ is $(d(n)\cdot N')$-almost edge-simple, that is, it travels at most $d(n)\cdot N'$ topological edge of $\hat\Delta$ more than once. Note also that $b_1(\hat\Delta)\le n-1$. 
Since $(\hat\Delta,p_{\hat\Delta})$ is a core pair where $p_{\hat\Delta}$ is the image of $p_\Delta$ in $\hat\Delta$, it now follows that $\hat\Delta$ has at most $3n-1$ maximal arcs. Therefore for any basic decomposition (!) we have $m\le 3n$ as distinct $p_i$ lie on distinct maximal arcs, with the possible exception of $p_1$ and $p_m$ lying on the same maximal arc.

The number of tuples $(p_1,\ldots ,p_m)$ (with each $p_t$ is understood as a subword of $v_j'$ occurring in a specific position in $v_j'$) arising in basic decompositions of a given $v_j'$, is bounded above by $f(n)$. This can been as follows: Let $\tilde\Delta$ be the graph obtained from $\hat\Delta$ by replacing maximal arcs  by edges, $\tilde\Delta$ has no vertex of valence~$2$.  Now any tuple $(p_1,\ldots ,p_m)$ gives rise to a path $e_1,\ldots ,e_m$ in $\tilde\Delta$ where $e_i$ is the edge obtained from the arc on which $p_i$ was read by the basic path.  This path satisfies the condition given in the definition of $f(n)$, and therefore there are at most $f(\tilde\Delta)\le f(n)$ different paths that occur. Each such path $(e_1,\ldots ,e_m)$ determines a unique tuple $(p_1,\ldots ,p_m)$: That $p_i$ is determined by $e_i$ is obvious for $2\le i\le m-1$ as $p_i$ is simply the label of the maximal arc of $\Delta$ corresponding to $e_i$. For  $i=1,m$ this follows from the fact that subwords of length at least $d(n)|v_i'|$ occur in only one position of the word $v_i'$. This shows that at most $f(n)$ tuples $(p_1,\ldots ,p_m)$ can occur.

We now show that the total length of all $y_t$ occurring in any fixed basic decompositions of $v_j'$ is at most $((1-\alpha)+6nd(n))\cdot |v_j'|$.


Choose a fixed basic decomposition (!) of $v_j'$ and let $\beta$ be a basic path representing (!). 
Each $y_t$ decomposes as a concatenation of \emph{arc subwords} that are read along maximal arcs in $\Psi'$ or the paths corresponding to $y_0'$ or $y_m''$ as $\beta$ is read in $\Delta$. 
Note further that the sum of the lengths of $y_0,\dots, y_m$ exceeds the number of edges in their image by at most $(1-\alpha)|v_j'|$ as $v_j'$ is $\alpha$-injective.  Each maximal arc in $\Psi'$  has length $< d(n)|v_j'|$ and so do the paths corresponding to $y_0'$ and $y_m''$. Thus the number of edges in the image of the arcs corresponding to $y_0,\dots, y_m$ is bounded from above by $6nd(n)|v_j'|$, since $\Delta$ has at most $6n-1$ maximal arcs and if $y_0'$ or $y_m''$ are non-degenerate then at most $6n-2$ of these arcs lie in $\Psi'$. Therefore 
 \[
 \sum_{t=0}^m |y_t| \le ((1-\alpha)+6nd(n))\cdot |v_j'|.
 \]

By definition of $f(n)$, it follows that the total length of all $y_t$ occurring in all possible basic decompositions of a given $v_j'$ is at most $$f(n)\cdot ((1-\alpha)+6nd(n))\cdot |v_i'|\le \frac{1}{2}|v_j'|.$$

Thus there exists a non-trivial subword $\tilde v_j$ of $v_j'$ (we
could actually choose $\tilde v_j$ to be a 1-letter subword) such that
in no basic decomposition of $v_j'$ this subword $\tilde v_j$
overlaps with any $y_t$. Then the subwords $\tilde v_1,\dots, \tilde
v_n$ satisfy  the requirements of Lemma~\ref{L:Claim}
\end{proof}

Recall that in Definition~\ref{def:q-injective} we defined the notions of an edge/arc of a graph $\Upsilon$ being \emph{$\Upsilon$}-injective with respect to a graph map from $\Upsilon$ to another graph. Recall also that in our situation the map  $f:\Gamma\to R_n$ factors
through $f_{\Delta}:\Gamma\to\Delta$. In the following
$\Upsilon$-injectivity of some arc $\Upsilon$ in $\Gamma$ refers to
$\Upsilon$-injectivity with respect to the map $f_{\Delta}$. Recall that $p_\Delta=f_\Delta(p_\Gamma)$ is the image of the base-vertex of
$\Gamma$ in $\Delta$.

Recall that, as specified in Convention~\ref{conv:Z},  the $(k+1)$-st loop of $\Gamma$
  contains an arc $\tilde s_Z$ labelled by a $U$-word $Z$ of length at
  least $(1-\frac{1}{10^4n})N^2$, and that the path $s_Z$ in $\Delta$ is the image of $\tilde s_Z$ under $p_\Delta$.

\begin{lem}\label{wecanshorten} The arc $\tilde s_Z$ contains a
  $\tilde s_Z$-injective subarc $Q$ of length  $\frac{1}{100n}N^2$ that
  is mapped to an arc $Q'$ in $\Delta$ such that $Q'$ does not contain
  the vertex $p_{\Delta}$ in its interior.
\end{lem}

\begin{proof}

Let $\tilde v_1, \dots, \tilde v_n$ be distinguished subwords of
$v_1',\dots v_n'$ provide by Lemma~\ref{L:Claim}.

\smallskip Let now $\Theta$ be the subgraph of $\Delta$ consisting of the edges traversed by $s_Z$. 


We now construct a graph $\hat\Theta$ from $\Theta$ as follows:

First, collapse all edges of $\Theta$ that do not lie on the arcs that are
images  of the subpaths of $s_Z$ corresponding to
distinguished subwords $\tilde v_j$ of the $v$-syllables of $Z$.  The resulting graph $\Theta'$ is the union of arcs labeled with the $\tilde v_j$ and two such arcs intersect at most in the endpoints.

Now in $\Theta'$ replace any such arc with label $\tilde
v_j$ with an edge with label $b_j$. We thus obtain a graph $\hat \Theta$
with edges labelled by letters of $\{b_1^{\pm 1},\dots, b_n^{\pm 1}\}$. Note that
$\hat\Theta$ has first Betti number at most $n-1$, since all edges of
$\Theta$ that lie in $\Psi$ are collapsed when constructing
$\Theta'$.

Since $Z$ is subword of length at least $\left(1-\frac{1}{10^4n}\right)N^2$  of some cyclic permutation of some $U_i^{\pm 1}$, the word $Z$ contains a subword $Z'$ with the following properties:
\begin{enumerate}
\item $|Z'|\ge \frac{49}{100}N^ 2$.
\item $Z'$ consists of the remnants of the  $v_i$ after substituting the $b_i$  in $u_i$.
\end{enumerate}

Let $s_{Z'}$ be the subpath of $s_Z$ corresponding to the subword $Z'$ of $Z$. Similarly let $\tilde s_{Z'}$ be the subpath of $\tilde s_Z$ corresponding to $Z'$.

Now for $j=1,\dots, n$ replace in $Z'$ any remnant of $v_j^\varepsilon$ by $b_j^\varepsilon$. We thus
obtain a freely reduced word $\hat Z$ in the $\{b_1^{\pm 1},\dots, b_n^{\pm 1}\}$ that is a subword of $u_i^{\pm 1}$. Note that $|\hat Z|\ge \frac{49}{100}N$.

The word $\hat Z$ is readable in $\hat\Theta$ by construction. Hence, by Theorem~\ref{readingranndomwordsingraphs} and by making the appropriate genericity assumptions on the $u_i$, we may assume that the corresponding path $s_{\hat Z}$ in $\hat \Theta$ is $\alpha$-injective for some $\alpha>\frac{1000n-1}{1000n}$. 
It follows from $\alpha$-injectivity of $s_{\hat Z}$ that any subpath of $s_{\hat Z}$ of length $$2(1-\alpha)\cdot |\hat Z|+1\le (1-\alpha)\cdot3N$$ contains an $s_{\hat Z}$-injective edge.
This implies that any subpath of $s_{Z'}$  of length $$((1-\alpha)\cdot3N+1)N\le (1-\alpha)\cdot4N^2=\frac{1}{250n}N^2\le \frac{1}{100n}N^2$$
contains an $s_{Z'}$-injective edge.

Now write $s_{Z'}$ as the the product of $20n-1$ subpaths $s_1,\ldots ,s_{20n-1}$ of length greater or equal than $\frac{1}{50n}N^2$.  This can be done since $|s_{Z'}|\ge \frac{49}{100}N^2$. By the argument given above, it follows that every $s_j$ contains an $s_{Z'}$-injective edge.

Since $\Delta$  has at most $3(2n-1)-3=6n-6$ maximal arcs, it follows
that the  the collection of all $s_{Z'}$-injective edges in $\Delta$
is the union of at most  $6n-5$ $s_{Z'}$-injective arcs of
$\Delta$.  At most one of these $s_{Z'}$-injective arcs of
$\Delta$ can contain $p_\Delta$ in its interior. After possibly subdividing
this arc by  $p_\Delta$ in two, it follows that  the collection of all
$s_{Z'}$-injective edges in $\Delta$ is the union of pairwise
non-overlapping arcs $t_1,\ldots ,t_k$, with $k\le 6n$ and such that $p_\Delta$ does not lie in the interior of any $t_i$.

Since each subpath $s_i$ of $s_{Z'}$ contains an $s_{Z'}$-injective
edge, it follows that each $s_i$ nontrivially overlaps some
$t_j$. We claim that there exist $k_0$ and $j$ such that each of
$s_{k_0-1},s_{k_0},s_{k_0+1}$ nontrivially overlaps $t_j$.  Otherwise
each $t_j$ would overlap at most two of the paths $s_1, \dots,
s_{20n-1}$, and hence the paths $t_1,\ldots ,t_k$ overall at most
$2\cdot 6n =12n < 20n-1$ of the paths $s_1, \dots, s_{20n-1}$. This
contradicts the fact that each $s_i$ intersects some $t_j$. Thus
indeed there exist $k_0$ and $j$ such that each of
$s_{k_0-1},s_{k_0},s_{k_0+1}$ overlaps $t_j$, so that $t_j$ contains
$s_{k_0}$. Note that $s_{k_0}$ does not contain the vertex $p_\Delta$
in its interior, since $p_\Delta$ does not belong to the interior of $t_j$.

Thus $s_{k_0}$ is an arc of $\Delta$ which is $s_{Z'}$-injective and
which has length $\ge \frac{1}{50n}N^2$.

By genericity of presentation $(*)$ we may assume that there is no
word $W$ of length $\frac{1}{500n}N^2$ such that $W^{\pm 1}$ occurs
more than once in $Z$. Since $s_{k_0}$ is an arc of length  $\ge
\frac{1}{50n}N^2$ and since $Z$ is freely reduced, the path $s_Z$ cannot run over the entire arc $s_{k_0}$ more than once.
We also know that $s_{k_0}$ does occur as a subpath of $s_{Z}$, so that at some point the path $s_Z$ does run over the entire arc $s_{k_0}$.
It may also happen that, in addition, the initial and/or terminal segment of $s_Z$ overlaps the arc $s_{k_0}$ nontrivially; however such overlaps must have length $\le \frac{1}{500n}N^2$.
Therefore $s_{k_0}$  contains a sub arc $Q'$ of length $\frac{1}{100n}N^2$ such that $Q'$ is $s_Z$-injective.

We put  $Q$ be the lift of $Q'$ to $\Gamma$.  Since $|Q|=|Q'|=\frac{1}{100n}N^2$, the conclusion of
Lemma~\ref{wecanshorten}  holds with these choices of $Q$ and $Q'$. Thus Lemma~\ref{wecanshorten} is established.
\end{proof}

\begin{proof}[Concluding the proof of Theorem~\ref{thm:main}]

We now have all tools to conclude the proof of the
Theorem~\ref{thm:main}. Recall that we argue by contradiction and that
we have assumed that the $2n-1$-tuples $(a_1,\dots, a_n,1\dots, 1)$
and $(b_1,\dots, b_n,1\dots, 1)$ are Nielsen-equivalent in $G$.

We will see that we can  apply Lemma~\ref{reductionmove} to reduce the complexity of $x_{k+1}$ while maintaining the complexity of $x_i$ for $1\le i\le k$ and preserving the fact that $x_i=_Gb_i$ for $1\le i\le n$.  This will yield is a contradiction to the minimality assumption about $(x_1,\dots, x_{2n-1})$ made at the beginning of the proof.

Recall that $k$ was maximal such that $x_i=v_i$ in $F(A)$ for all $1\le i\le k$ and that
by Lemma~\ref{lem:k<n} we know that $k\le n-1$. By convention if
$x_1\ne v_1$, we set $k=0$. Thus we always have $1\le k+1\le n$.

Let $W$ be the word that is the label of the arc $Q$ of length
$\frac{1}{100n}N^2$ provided by Lemma~\ref{wecanshorten}.  Thus $W$ is of length $\frac{1}{100n}N^2$. Note that $W$ is a subword of the $U$-subword $Z$ of length at least $\left(1-\frac{1}{10^4n}\right)N^2$.
Recall that $Q$ is a sub-arc of the $(k+1)$-st petal of $\Gamma$ and
that under the map $\Gamma\to\Delta$ the arc $Q$ maps injectively to
an arc $Q'$ of $\Delta$.

By
Lemma~\ref{wecanshorten} the base-vertex $p_\Delta=f_\Delta(p_\Gamma)$ does
not belong to the interior of $Q'$. Since the petals of $\Gamma$ are
labelled by freely reduced words, it now follows that the full
preimage of $Q'$ in $\Gamma$ under $f_\Delta$ is the union of a
collection of pairwise non-overlapping arcs $Q_1,\dots, Q_m$ in
$\Gamma$ such that $Q=Q_1$ and such that each of the arcs $Q_1,\dots,
Q_m$ maps bijectively to $Q'$.

For $1\le i\le 2n-1$ write $x_{i}$ as a reduced product

\[
x_{0,i}W^{\varepsilon_{1,i}}x_{1,i}W^{\varepsilon_{2,i}}\cdot\ldots
W^{\varepsilon_{q_i,i}}x_{q_i,i}\tag{$\ddag$}
\]
 with $\varepsilon_{j,i}\in\{-1,1\}$
where the $W^{\varepsilon_{j,i}}$ are the occurrences of $W^{\pm 1}$ in
$x_i$ corresponding to all those of the arcs $Q_1,\dots, Q_m$ that are
contained in the $i$-th petal-loop of $\Gamma$. 

Note that for $1\le i\le k$ the length of $x_i=v_i$ is too short to
contain $W$ as a subword. Therefore for $1\le i\le k$ we have
$x_i=x_{0,i}$ and $q_i=0$.

We now perform the change described in Lemma~\ref{reductionmove}
simultaneously for all $x_i$, $i=1,\dots, 2n-1$. That is, we pick a
word $V\in F(A)$ such that $WV^{-1}$ is a cyclic permutation of one of
the defining relations $U_s^{\pm 1}$ of presentation (**). Then for
each $x_i$ we replace each $W^{\epsilon_{j,i}}$ in the decomposition $(\ddag)$
of $x_i$ by $V^{\epsilon_{j,i}}$. Denote the resulting word by $z_i$ and
denote the freely reduced form of $z_i$ by $x_i'$, where $i=1,\dots, 2n-1$.

By construction,  $x_i$ is unchanged for $1\le i\le k$, that is
$x_i=x_i'$ for $1\le i\le k$. Moreover the complexity of $x_{k+1}$
decreases as the presence of subword $Z$ of $x_{k+1}$ triggers clause (2) of Lemma~\ref{reductionmove}. Thus we have $c(x_i)=c(x_i')$ for $1\le i\le k$ and $c(x_{k+1}')<c(x_{k+1})$.

Since $W=_G V$, we have $x_i=_Gx'_i$ for $1\le i\le
2n-1$. In particular we have $x_i'=_Gb_i$ for $1\le i\le n$. Thus to
get the desired contradition to the minimal choice of $(x_1,\dots,
x_{2n-1})$ we only need to verify that the tuple $(x_1',\ldots
,x'_{2n-1})$  is Nielsen-equivalent to $(a_1,\ldots ,a_n,1,\ldots ,1)$
in $F(A)$. This fact follows from the following considerations:

In the graph $\Delta$ we replace the arc $Q'$ by an arc $Y$ with label $V$ and
call the resulting graph $\Delta'$. The graph $\Delta'$ still folds
onto $R_n$, since $\Delta'$ contains $\Psi$.

By
Lemma~\ref{wecanshorten} the vertex $p_\Delta=f_\Delta(p_\Gamma)$ does
not belong to the interior of $Q'$. Since the petals of $\Gamma$ are
labelled by freely reduced words, it now follows that the full
preimage of $Q'$ in $\Gamma$ under $f_\Delta$ is the union of a
collection of pairwise non-overlapping arcs $Q_1,\dots, Q_m$ in
$\Gamma$ such that $Q=Q_1$ and such that each of the arcs $Q_1,\dots,
Q_m$ maps bijectively to $Q'$. Note also that each of the arcs $Q_1,\dots, Q_m$
is labelled by $W$ amd that the base-vertex
$p_\Gamma$ of $\Gamma$ does not belong to the interior of any of
$Q_1,\dots, Q_m$.

We obtain a new graph $\Gamma'$ from $\Gamma$ by replacing each $Q_j$
by an arc $Y_j$ labelled by $V$. Then there is a natural label-preserving
graph map
$f_{\Delta'}:\Gamma'\to \Delta'$ which sends each arc $Q_j$ bijectively to
$Y$, and which agrees with $f_\Delta$ on the complement of $Y_1\cup
\dots \cup  Y_m$ in $\Gamma'$. The map $f_\Delta: \Gamma\to\Delta$
arose from a sequence of folds and therefore $f_\Delta$ is $\pi_1$-surjective.

Topologically $f_{\Delta'}$ can be
viewed as the same as the map $f_\Delta$ and hence the map
$f_{\Delta'}$ is $\pi_1$-surjective. As we noted earlier, the graph
$\Delta'$ folds onto $R_n$, and hence the map $\Delta'\to R_n$ is
$\pi_1$-surjective. By composing this map with $f_{\Delta'}$, we
conclude that the canonical map $\Gamma'\to R_n$ preserving
edge-labels is $\pi_1$-surjective.

The graph $\Gamma'$ is a wedge of $2n-1$ circles with labels
$z_1,\dots, z_{2n-1}$.  As noted above, the canonical map
$\Gamma'\to R_n$ is $\pi_1$-surjective. Since $\pi_1(\Gamma')$ is
generated by the loop-petals with labels $z_1,\dots, z_{2n-1}$ and
since each $z_i$ freely reduces to
$x_i'$, it follows that the elements $x_1',\dots, x_{2n-1}'$ of $F(A)$
generate the entire group $F(A)=F(a_1,\dots, a_n)$. 
Therefore the $(2n-1)$-tuple $(x_1',\dots, x_{2n-1}')$ is
Nielsen-equivalent in $F(A)$ to the $(2n-1)$-tuple $(a_1,\dots, a_n, 1,\dots, 1)$.

Recall that $x_i=v_i=x_i'$ for $i=1,\dots, k$ (where $k<n$) and that
$c(x_{k+1}')<c(x_k)$.
Therefore $c(x_1',\dots, x_{2n-1}')< c(x_1,\dots, x_{2n-1})$, which
contradicts the minimality assumption on $(x_1,\dots, x_{2n-1})$.

Hence the $2n-1$-tuples $(a_1,\dots, a_n,1,\dots, 1)$
and $(b_1,\dots, b_n,1,\dots, 1)$ are not Nielsen-equivalent in $G$,
and Theorem~\ref{thm:main} is proved.

\end{proof}

\end{document}